\documentclass{amsart}
\usepackage{amsmath}
\usepackage{amssymb}
\usepackage[all]{xy}
\usepackage{comment}
\usepackage{stackrel}
\usepackage{color}
\usepackage{url}

\theoremstyle{plain}
\newtheorem{thm}{Theorem}[section]
  
\newtheorem{thm*}{Theorem}[section]
\newtheorem{cor}[thm]{Corollary}
\newtheorem{prop}[thm]{Proposition}
\newtheorem{lemma}[thm]{Lemma}

\newtheorem{lemma*}{Lemma}

\newtheorem{proposition}[thm]{Proposition} 
\newtheorem{corollary}[thm]{Corollary}

\theoremstyle{definition}
\newtheorem{defn}[thm]{Definition}
\newtheorem{remark}[thm]{Remark}

\newtheorem*{remark*}{Remark}

\newtheorem{ex}[thm]{Example}

\newtheorem{question*}{Question}

\numberwithin{equation}{thm}

\newcommand{\bC}{\mathbb C}

\newcommand{\cW}{\mathcal W}

\def\rk{\operatorname{Rk}\nolimits}

\def\dim{\operatorname{dim}\nolimits}

\def\pic{\operatorname{Pic}\nolimits}

\def\im{\operatorname{im}\nolimits}

\DeclareMathOperator{\codim}{codim}

\DeclareMathOperator{\Hom}{Hom}
\DeclareMathOperator{\spec}{Spec}

\DeclareMathOperator{\id}{id}
\DeclareMathOperator{\charct}{char}

\DeclareMathOperator{\contalg}{c.alg}

\newcommand{\ol}{\overline}

\newcommand{\cO}{\mathcal O}
\newcommand{\bA}{\mathbb A}

\newcommand{\cF}{\mathcal F}
\newcommand{\bP}{\mathbb P}
\newcommand{\bZ}{\mathbb Z}

\newcommand{\cZ}{\mathcal Z}

\newcommand{\cC}{\mathcal C}

\newcommand{\cK}{\mathcal K}

\newcommand{\cal}{\mathcal}

\newcommand{\bb}{\mathbb}
\newcommand{\PP}{\mathbb{P}}

\def\rk{\operatorname{rk}\nolimits}

\def\id{\operatorname{id}\nolimits}

\def\sl2{\operatorname{SL_{2(2)}}\nolimits}
\def\Ga2{\operatorname{\mathbb G_{\rm a(2)}}\nolimits}

\def\Hom{\operatorname{Hom}\nolimits}

\def\colim{\operatorname{colim}\nolimits}

\def\FRAC{\operatorname{Frac}\nolimits}

\newcommand{\wt}{\widetilde}
\newcommand{\xym}{\xymatrix}

\newcommand{\Z}{\mathbb Z}

\newcommand{\bu}{\bullet}

\newcommand{\into}{\hookrightarrow}

\newcommand{\Deltadot}{\Delta^{\bullet}}

\date\today

\begin{document}

 \title[intersections via motivic complexes]{An approach to intersection theory on singular varieties using motivic complexes}
 
 \author{Eric M. Friedlander$^{*}$ and Joseph Ross}

\address {Department of Mathematics, University of Southern California,
Los Angeles, CA}
\email{ericmf@usc.edu, eric@math.northwestern.edu}

\address {Department of Mathematics, University of Southern California,
Los Angeles, CA}
\email{joseph.ross@gmail.com}

\thanks{$^{*}$ partially supported by the NSF grants DMS-0909314 and DMS-0966589}

\subjclass[2000]{14F43, 14C25, 19E15}

\keywords{intersection {homology}, perversity, cocycles, {cycle complexes}}

\begin{abstract} 
We introduce techniques of Suslin, Voevodsky, and others into the
study of singular varieties.  Our approach is modeled after Goresky-MacPherson 
intersection homology.  We provide a formulation of  perversity cycle spaces
leading to perversity homology theory and a companion perversity cohomology 
theory based upon generalized cocycle spaces.  These theories lead to 
conditions on pairs of cycles which can be intersected and a suitable 
equivalence relation on cocycles/cycles enabling pairings on
equivalence classes.  We establish suspension and splitting theorems, 
as well as a localization property.  Some examples of intersections on singular 
varieties are computed.
\end{abstract}

\maketitle

\tableofcontents

%%%%%%%%%%%%%%%%%%%
%%%%%%%%%%%%%%%%%%%%

\section{Introduction}

	In this paper, we initiate an investigation of pairings on cycle
groups on singular algebraic varieties over a field.  We 
utilize the approach to motivic cohomology developed by A.~Suslin and V.~Voevodsky \cite{SV},
blended with the philosophy of intersection homology theory as 
introduced by M.~Goresky and R.~MacPherson \cite{GM1}.  An important source of insight
for the approach we take comes from ``semi-topological cohomology and 
homology," especially from the foundations of Lawson homology due to H.~B.~Lawson \cite{Lawson}.

	Our goal is to provide contexts in which there is a good formulation
of the intersection product of cycles on singular varieties.  This is an
age-old problem, one that motivated the original introduction of cohomology
and in some sense culminated with intersection homology theory for 
stratified topological spaces.  In the context of algebraic varieties, the
moving techniques for stratified spaces (for example, those of \cite{McC})
do not apply.  Indeed, we know of
no means of improving intersections occurring within the singular locus of
a given variety.

{If  $X$ is a smooth projective variety, then Poincar\'e duality provides a ring structure
on the singular homology of $X$.  This product admits a purely algebro-geometric description
on the fundamental classes of algebraic cycles $\alpha$ and $\beta$: by the method of Chow, one can move $\alpha$
within its rational equivalence class (to $\alpha'$, say) so that $\alpha'$ and $\beta$ intersect properly (i.e., in the expected dimension).
For proper intersections on a smooth variety, multiplicities may be defined purely algebraically,
for example by the Tor-formula of Serre.  The homology class of the cycle class $\alpha \bu \beta$ represents the product of the homology
classes of $\alpha$ and $\beta$. }

{If $X$ is singular, then its homology groups typically cannot be endowed with a reasonable ring structure.  
The intersection homology of Goresky-MacPherson rectifies this by defining groups $IH_*^{\ol p}(X)$ 
which, roughly speaking, are the homology groups of a complex of chains with controlled incidence 
with the singular locus of $X$.  There are intersection pairings 
$IH^{\ol p}_r (X) \otimes IH^{\ol q}_s(X) \to IH^{\ol p + \ol q}_{r+s - \dim(X)}(X)$ (provided some conditions 
are satisfied) which, {in case $r+s = \dim(X)$,} become perfect after tensoring with the rationals.  {The challenge which originally motivated us }
was to extend the picture of the previous paragraph, namely the description of the intersection product of 
algebraic cycle classes, to intersection homology of singular varieties. }

{Previous approaches to this problem have not led to an intersection pairing lifting the Goresky-MacPherson pairing.
P.~Gajer defined a semi-topological version of intersection homology and established
some of its structural properties \cite{Gajer}.  A.~Corti and M.~Hanamura gave a definition of intersection Chow groups by incorporating
information obtained from a resolution of singularities \cite{CortiHan2}; they provided also a motivic lifting of the decomposition theorem of
\cite{BBD} assuming various conjectures on algebraic cycles \cite{CortiHan1}.  J.~Wildeshaus used weight structures to define a motivic intersection complex, and proved its existence in some cases \cite{Wild}.  In the topological setting, intersection homology may be defined geometrically, using a subcomplex of the complex of singular chains \cite{GM1}, or sheaf-theoretically, using the constructible derived category \cite{GM2}.  In the algebraic setting, it would be interesting to relate our geometrically oriented approach to the categorical constructions.}

Introducing cycle (and cocycle)
spaces and defining homotopy pairings on these spaces guides the formulation 
of equivalence relations on cycles and gives pairings on homotopy groups.  
The equivalence relations which arise are necessarily finer than rational equivalence:
even if one restricts attention to cycles which meet
``properly" and whose intersection meets the singular locus properly, one 
must take care in defining equivalence relations so that cup and cap product 
pairings are well defined on equivalence classes. 
Our primary interest is the intersection of fundamental classes
of algebraic cycles, corresponding to a pairing on connected components of
our cycle and cocycle spaces.

	We work with an algebraic variety $X$ equipped with a stratification; 
such a stratification might arise from  a resolution of singularities of $X$
or a ``platification" of a family of coherent sheaves on $X$.  
Fixing a perversity function $\ol p$, we introduce 
perversity cycles on $X$ and generalized cocycles on $X$ with values in $Y$.
These are cycles which meet the strata of $X$ (or
$X\times Y$) in a manner controlled by $\ol p$.   The discrete abelian 
groups of perversity cycles (and generalized cocycles) for a given variety
$X$ determine  presheaves which lead to singular
complexes (i.e., simplicial abelian groups) as first conceived by 
Suslin (see \cite{SV}).  Our homology/cohomology theories are the homotopy groups of 
these singular complexes, doubly graded in a manner compatible with the 
grading in motivic homology and cohomology \cite{MVW}. 

	We show that our theories satisfy good properties including suspension
 isomorphisms {({a projective analogue of }$\bb{A}^1$-invariance)}, a splitting theorem, and a suitable form of
localization.  These theorems enable our definition of a cup product
in perversity cohomology, extending the cup product in motivic cohomology;
{and a cap product relating perversity cohomology and perversity homology.}
To do this, we introduce the condition $(*, {\ol c})$ on
a pair of cycles and a perversity $\ol c$ which permits a sensible intersection
of cycles meeting especially nicely;  this intersection product is compatible with
that of Goresky-MacPherson intersection homology.

For the reader's convenience, we briefly outline the contents of each section of 
this paper.

In Section \ref{sec:one}, we revisit various sheaves and presheaves of
relative cycles as investigated by Suslin and Voevodsky.  We discuss to 
what extent and how these sheaves are represented by Chow varieties. These (pre)sheaves
are defined on $(Sch/k)$ so that we may apply results of Voevodsky on sheaves
for the cdh-topology;
the cycle sheaves are evaluated on the standard cosimplicial scheme
whose constituents are affine spaces $\Delta^n$.   In fact, one is naturally led to 
another of Voevodsky's Grothendieck topologies, the h-topology, when the 
characteristic of the ground field is positive.

We begin our study of cycles on a stratified (possibly singular) variety $X$ in Section \ref{sec:two}.
Following Goresky-MacPherson, we fix a ``perversity" $\ol p$ and consider $U$-relative
cycles on $U\times X$ whose specializations at points $u \in U$ meet the strata of $X_u$
in codimension controlled by the perversity $\ol p$.  Applying our sheaves to
$\Deltadot$, we obtain our perversity motivic homology groups 
$H^{\ol p}_n(X,\bZ(r))$ as the homotopy groups of the associated simplicial abelian
group (or, equivalently, as the homology of the associated chain complex).   There is
a natural map to motivic Borel-Moore homology
$H^{\ol p}_n(X,\bZ(r)) \to H^{BM}_n(X,\bZ(r))$
induced by an inclusion of simplicial abelian groups.
Furthermore, when our ground field
$k$ is the complex field $\bC$,  we verify in Proposition \ref{prop:GM} that
there is a natural map from the bidegree
perversity homology group corresponding to $\pi_0$ to 
 the Goresky-MacPherson intersection homology group.
  In Theorem \ref{cdh ses}, we use techniques of Voevodsky  to prove a form of localization
for our perversity motivic homology groups.  

A central theme of our work is the interplay between the sheaf-theoretic 
foundations of Suslin-Voevodsky and constructions using Chow varieties as 
first considered by Lawson in \cite{Lawson}.  In particular, in Section \ref{sec:three},
we employ the constructions introduced by Lawson to prove suspension
theorems for our homology groups.  These theorems are first proved in Theorem \ref{thm:proj-susp} for 
projective varieties (for Chow varieties are defined for projective varieties) and 
then extended to quasi-projective varieties using the localization theorem of 
the previous section.  The proofs require verification that ``Lawson moving constructions"
preserve perversity of cycles.

In Section \ref{sec:four}, we relate our groups to the problem of intersecting cycles 
on a stratified singular variety.  We introduce 
the condition $(*, {\ol c})$ on a pair of cycles which allows (static) intersection with
good properties, especially suitable behavior with respect to specialization.
{For example, Corollary \ref{specialization} verifies that this intersection
commutes with specialization, the formal analogue of being continuous.}
We analyze in detail the resulting intersection pairing for the  standard 
example due to Zobel of the cone on $\bP^1 \times \bP^1$.  

The ``generalized cocycles" introduced in  Definition \ref{def:cocycle} pair 
well with perversity cycles.  Perversity 
cycles satisfy an incidence condition
with the strata of a given stratified variety $X$, whereas 
generalized cocycles on $X$ with values in some $Y$ are cycles on 
$X\times Y$ more general than cocycles (i.e., not necessarily equidimensional
over $X$) whose fiber dimensions over points of $X$ are controlled by 
the perversity $\ol p$.  Algebraic cocycles
first appeared in work of the first author and Lawson \cite{FLcocycle} as an algebraic
model for cocycles in algebraic topology; our groups are a stratified variant 
of groups briefly considered by the first author and Gabber in \cite{FG} 
(a more sophisticated form of which is presented in the paper of the first author
and Voevodsky \cite{FV}.)  These bivariant perversity motivic cohomology
groups satisfy a suspension theorem (Theorem \ref{thm:suspgen})
which leads to perversity motivic cohomology
by setting the covariant variable equal to a projective space.  As we describe,
generalized cocycles arise from resolutions and 
from coherent sheaves (with stratification determined by the resolution or the sheaf).
{The second author has investigated the possibility of using resolutions
to define an intersection theory for perversity cycle classes on singular varieties \cite{JoeNew}.
The geometric approach of \cite{JoeNew} produces pairings in several interesting cases.}

In the final section of this paper, we lay the foundations for applications by
establishing a cup product on perversity motivic cohomology and a cap
product pairing relating perversity motivic cohomology and perversity motivic homology.
{For example, in Theorem $\ref{thm:splitting}$
we establish the (motivic) perversity version of the 
splitting theorems established for semi-topological cohomology by the first
author and Lawson.}
We conclude by verifying in Proposition \ref{prop:new compat} that our
cup product pairing is compatible with the intersection product in intersection homology.

Throughout, we work over an infinite field $k$ of characteristic $p \geq 0$.
When we apply Voevodsky's acyclicity theorem, we must additionally impose that $k$ is perfect.
In positive characteristic, we try to avoid inverting $p$ wherever possible.
For us, a $k$-scheme is a separated scheme of finite type over $k$, and a 
variety is an integral $k$-scheme.

%%%%%%%%%%%%%%%%%%%
%%%%%%%%%%%%%%%%%%%%%%%%%%%

\section{Roadmap for cycle presheaves}
\label{sec:one}

We employ a  plethora of presheaves and sheaves of algebraic cycles.  Our invariants are homotopy groups of
simplicial abelian groups { (equivalently, homology groups of associated normalized chain complexes)
obtained by evaluating an abelian (pre)sheaf on a cosimplicial scheme.}
Our geometric constructions are correspondences among Chow varieties of $r$-dimensional 
cycles on a projective variety $X$.  { The presheaves represented by Chow varieties are
closely related to the Suslin-Voevodsky presheaves $z_{equi}(X,r)$ and $z(X,r)$.  
To extend our results to quasi-projective $X$, we employ the technology developed by Suslin and Voevodsky for sheaves for the
cdh-topology. }

{For a scheme $X$,} 
the Suslin-Voevodsky cdh-sheaf $z(X,r)$ on $(Sch/k)_{cdh}$ sends a $k$-scheme $U$
to  the abelian group of $U$-relative cycles on $U \times X$ (of relative dimension $r$) with well-defined 
specializations and universally integral coefficients \cite[Lemma 3.3.9]{SV}.   {If $k$ admits resolution of singularities,}
this sheaf has the 
important {\it localization property} (see \cite[Thm.~4.3.1]{SV}, \cite[Remark 5.10]{FV}): 
if $Y \into X$ is closed with Zariski open complement $U$, then the triple of simplicial abelian groups
$$z(Y,r)(\bu) \ \to \ z(X,r)(\bu) \ \to \ z(U,r)(\bu)$$
yields a long exact sequence of homotopy groups, where $\cal{F}(\bu) \equiv \cal{F}(\Deltadot)$ is the simplicial abelian
group whose associated chain complex is (by definition) the Suslin complex of $\cal{F}$.

Assume now that $X$ is projective and consider the  subsheaf 
$z^{eff}(X,r) \ \subset \ z(X,r)$ whose value on $U$ is the monoid of those $U$-relative cycles which are effective.
When $\charct(k) = 0$, cycles in $z^{eff}(X,r)(U)$ can be identified with the 
graphs of homomorphisms from the semi-normalization of $U$ into the Chow monoid $\cC_r(X)$,
\begin{equation}
\label{eq:eff}
z^{eff}(X,r)(U) \ \simeq \Hom(U^{sn},\cC_r(X)) \ \simeq \Hom(U^{sn}, {\cC_r(X)}^{sn}), \quad \charct(k) = 0;
\end{equation}
this is the h-representability of the sheaf $z^{eff}(X,r)$ \cite[Cor.~4.4.13]{SV} and the fact that $z^{eff}(X,r) \to z^{eff}(X,r)_h$ and $z(X,r) \to z(X,r)_h$ are isomorphisms in characteristic zero \cite[Thm.~4.2.2]{SV}.

In arbitrary characteristic, the h-sheafifications may be computed using continuous algebraic maps (\cite[4.1]{F94}, \cite[Cord.~4.4.13]{SV}), so that
$$z^{eff}(X,r)_h(U) \ \simeq \Hom(U,\cC_r(X))_h \ \simeq  \Hom_{\contalg}(U, \cC_r(X)).$$
(This h-sheafification admits a description as a limit of morphisms of schemes even though in positive characteristic the object h-representing a sheaf is not unique; see \cite[Prop.~3.2.11]{VHomSch}.)  Notice also that $p$ is invertible in $\Hom_{\contalg}(U, \cC_r(X))$ if $U$ is equidimensional, since then the relative Frobenius $F_{U/k} : U \to U^{(1)}$ is generically flat of degree $p^{\dim_k(U)}$, and the continuous algebraic map $U^{(1)} \xleftarrow{F_{U/k}} U \xrightarrow{f} \cC_r(X)$ corresponds to $1/ p^{\dim_k(U)} \cdot f$.  Since $z^{eff}(X,r)[1/p]$ and $z(X,r)[1/p]$ are h-sheaves \cite[Thm.~4.2.2]{SV}, this implies $z^{eff}(X,r)_h(U)[1/p] \ \simeq  \Hom_{\contalg}(U, \cC_r(X))$ for $U \in Sm / k$.  

The presheaf $z^{eff}(X,r)$ admits a reasonable description in terms of Chow varieties before inverting $p$.
If $X$ is projective and $\charct (k) = p > 0$, then for $U$ smooth and quasi-projective, the subgroup
$z^{eff}(X, r) (U) \subset \Hom(U, \cC_r(X))$  consists of those morphisms $f : U \to \cC_r(X)$ 
such that, for every generic point $\eta \in U$, the cycle 
classified by $f(\eta)$ is defined over $k(f(\eta))$; in general, the cycle classified
 by $f(u)$ is defined over a finite radicial extension of $k(f(u))$ \cite[Prop.~2.3]{FV}.  For details on the concept of field of definition of a cycle, we refer the reader to \cite[Defn.~I.3.1.7]{Kol}; this reference also contains examples of what can go wrong in positive characteristic (see for example \cite[Ex.~I.4.1.1]{Kol}).

\medskip
Now we consider possibly ineffective cycles, retaining the hypothesis that $X$ be projective.
A subtlety arises in comparing the presheaf $z^{eff}(X,r)^+$ to the sheaf $z(X,r)$, one that arises
because not every element of $z(X,r)(U)$ is a difference of elements of $z^{eff}(X,r)(U)$.  
In general, there is an intermediate presheaf 
$$ z^{eff}(X,r)^+ \ \subset \ z_{equi}(X,r) \ \subset z(X,r)$$ 
consisting of $U$-relative cycles on $X$ each component of which has relative dimension $r$.
Examples show that $z_{equi}(X,r)(U)$ can strictly contain  $z^{eff}(X,r)^+(U)$ and be strictly
contained in $z(X,r)(U)$.  Nevertheless, by \cite[Cor.~3.4.4]{SV} we have 
$$ z^{eff}(X,r)^+(U) \ = z_{equi}(X,r)(U), \quad U \ \text{geometrically unibranch}.$$ 
Consequently, 
$$z_{equi}(X,r)_{|_{(Sm/k)}} \ \subseteq \ (\Hom(-, \cC_r(X))^+)_{|_{(Sm/k)}}$$
with equality if $\charct(k) = 0$ and with image consisting of morphisms satisfying the field of definition condition described above
if $\charct(k)= p > 0$.

The cycle sheaves $z(X,r)$ and $z(X,r)_h$ for $X$ projective can be described in terms of continuous algebraic maps 
to the group completion $\cZ_r (X) := \cC_r(X)^2/R$ of the Chow monoid $\cC_r(X)$; here, $R$ is the usual relation
\ $(V,W) \sim (V^\prime,W^\prime)$ if and only if $V + W^\prime = W + V^\prime$ (\cite[4.1]{F94}, \cite[Prop.~4.4.15]{SV}).  
We remind the reader that a continuous algebraic map to the group completion is (up to a bicontinuous algebraic map) 
a pair of \textit{rational} maps to the Chow monoids which induces a well-defined set-theoretic map (on $\ol k$-points) 
to $\cZ_r(X_{\ol k})$.  This permits fibers of dimension $>r$ (which may occur outside the domains of definition of the rational maps) 
to cancel.  Then as sheaves on $(Sch/k)$, we have (\cite[4.1]{F94}, \cite[Prop.~4.4.15]{SV}):
$$z(X,r) \ = \ z(X,r)_h \ = \  \Hom_{\contalg}(-,\cZ_r (X)), \quad \charct(k) = 0$$
\vspace{-.4cm}
$$z(X,r)[1/p] \ = \ z(X,r)_h \ = \ \Hom_{\contalg}(-, \cZ_r (X))[1/p], \quad \charct(k) = p.$$

{By the above, we have $\Hom_{\contalg}(U, \cZ_r (X))[1/p] =  \Hom_{\contalg}(U, \cZ_r (X))$.}
Furthermore, for $ \charct(k) = p$ and $U \in Sm / k$, the image $z(X,r)(U) \ \subset \ \Hom_{\contalg}(U,\cZ_r(X))$ consists of those 
continuous algebraic maps $U \to \cZ_r(X)$ which
are induced by a pair of morphisms from an open dense subset of $U$ (i.e., the bicontinuous algebraic map is an isomorphism), 
both of which satisfy the field of definition condition.

\medskip 
What ties all this together is a fundamental acyclicity result of Voevodsky \cite[Thm.~5.5(2)]{FV} which asserts 
that the map of presheaves $z_{equi}(X,r) \to z(X,r)$ induces a quasi-isomorphism on associated Suslin 
complexes provided $k$ admits resolution of singularities.  Thus, the localization property for $z(X,r)$
can be ``transported" to presheaves ``represented" by Chow varieties.  In positive characteristic, the same is true after
inverting the characteristic.  Using O.~Gabber's theorem on the existence of smooth alterations of degree prime to $\ell$ \cite[1.3]{Gabber1}, \cite[3.2.1]{Gabber2}, recent work of S.~Kelly establishes that if $k$ is a perfect field of exponential characteristic $p$, and $\cal{F}$ is a presheaf with transfers on $Sch / k$
such that $\cal{F}_{cdh}[1 / p] = 0$, then the Suslin complex $C_* (\cal F) [ 1 / p]$ is acyclic \cite[Thm.~5.3.1]{Kelly}.
Thus our methods apply to an arbitrary infinite perfect field $k$ once perversity homology and cohomology groups are tensored with $\bZ[1/p]$.

 In Section \ref{sec:perverse}, we consider a ``bivariant" version of these constructions.  Namely, we
 consider  quasi-projective varieties $X, \ Y$ of pure dimension $d, n$.  We have subpresheaves and subsheaves 
 $$z^{t, eff}(X,Y) \ \subset \ z^{eff}(X\times Y,d+n-t), \quad z^{t}(X,Y) \ \subset z(X\times Y,d+n-t)$$
 which guide us to various ``cohomological" theories on $X$ (taking $Y$ to be projective space).
 
 This paper is concerned with versions of these presheaves and sheaves for stratified varieties 
 and a given perversity.  Thus, the presheaves and sheaves we consider will be elaborations
 of the ones mentioned above, taking into account the stratification and perversity.

\section{Motivic theories with perversity}
\label{sec:two}

\textbf{Stratifications and perversities.}
We assume $X$ and $Y$ are equidimensional $k$-schemes of dimension $d$ and $n$ respectively.

A \textit{stratified variety} is a variety $X$ equipped with a filtration by 
closed subsets $X^d \into X^{d-1} \into \cdots \into X^2 \into X^1 \into X$ such that $X^i$ has 
codimension at least $i$ in $X$.
If $X$ and $Y$ are stratified, we say $f: Y \to X$ is a \textit{stratified morphism}
if $f$ is a morphism of schemes such that $f(Y^i) \subseteq X^i$ for all $i$.  
A \textit{perversity} is a non-decreasing sequence of integers 
$p_1, p_2, \ldots, p_d$ such that $p_1=0$ and, for all $i$, $p_{i+1}$ equals either $p_i$ or 
$p_i+1$.  Perversities are denoted $\ol{p}, \ol{q}$, etc.  The perversities we consider range 
from the zero perversity $\ol{0}$ with $p_i=0$ for all $i$, to the top perversity $\ol{t}$, 
with $p_i=i-1$ for all $i$.  Our convention differs from that of Goresky-MacPherson \cite[1.3]{GM1} since over the complex numbers
our strata always have even real dimension; our $p_i$ corresponds to their $p_{2i}$.

Let $Z_r(X)$ denote the group of $r$-dimensional algebraic cycles on $X$.  Suppose $X$ is stratified.
We say an $r$-cycle $\alpha$ is \textit{of perversity $\ol{p}$} (or satisfies the perversity condition 
$\ol{p}$) if for all $i$, the dimension of the intersection $|\alpha| \cap X^i$ is no larger than 
$r-i +p_i$.  {When the codimension of $X^i$ in $X$ is exactly $i$,}
the perversity of a cycle measures its failure to meet properly the closed sets 
occurring in the stratification of $X$.  Let $Z_{r, \ol{p}}(X) \subset Z_r(X)$ denote the 
group of $r$-dimensional cycles of perversity $\ol{p}$ on the stratified variety $X$.  
Often, $X^1$ is taken to be the  singular locus of $X$, and then the condition $p_1=0$ means 
that no component of the cycle is contained in the singular locus.

Since elements of $z (X,r)(U)$ are required to have well-defined specializations for $u \in U$, 
we may define subpresheaves by imposing incidence conditions on the fibers over all $u \in U$.  
Let $T$ be an irreducible locally closed subset in $X$ and $p$ an integer.
 For $U \in Sch /k $, we define $z(X,r)_{T,p}(U) \subseteq z(X,r)(U)$ to be the subgroup of $U$-relative 
cycles $\alpha \into  U \times X$ satisfying the additional condition that, for all $u \in U$, 
the intersection of the support of $\alpha_u$ with $T_{u}$ in $X_{u}$ has excess at most $p$.  
This condition is topological, hence insensitive to the field of definition of the various $\alpha_u$'s. 

If $f: U' \to U$ is a morphism in $Sch / k$ and $\alpha \in z(X,r)(U)$ is a cycle, then for 
all $u' \in U'$, by functoriality the cycle ${(f^* \alpha)}_{u'}$ coincides with the 
cycle ${(\alpha_u)}_{u'}$, where $f(u') = u$.  Since the morphism $f_{u'}: \spec (k(u')) \to \spec (k(u))$ 
is universally open, by \cite[Lemma 3.3.8(1)]{SV} the support of ${(\alpha_u)}_{u'}$ 
is the base change via $f_{u'}$ of the support of $\alpha_u$.  
Therefore the assignment $U \mapsto z(X,r)_{T,p}(U)$ defines a presheaf
$z(X,r)_{T,p}(-) \subseteq z(X,r)(-)$.
The behavior of supports under base change also implies that if $f : U' \to U$ is an h-cover and $\alpha \in z(X,r)(U)$ satisfies $f^* \alpha \in z(X,r)_{T, p}(U')$, then $\alpha \in z(X,r)_{T, p}(U)$.  Therefore $z(X,r)_{T, p} \subseteq z(X,r)$ is a cdh-subsheaf.

Similarly we may define a sheaf 
\begin{equation}
\label{tp}
z(X,r)_{ \mathcal T, p}(-) \ \subseteq  \ z(X,r)(-)
\end{equation}
where $\cal T$ is a collection of irreducible locally closed subsets  of $X$
and $p$ is a $\bb{Z}_{\geq 0}$-valued function on $\cal T$: we require
 the excess with $T \in \cal T$ to be bounded by $p(T)$.
We refer to such a pair $(\cal T, p)$ as an \textit{incidence datum} on $X$.
The equidimensional version is denoted $z_{equi}(X,r)_{\mathcal T, p}$.
 If $(\cal T, p)$ and $(\cal T', p')$ are incidence data with $\cal T \subseteq \cal T'$ and $p' |_{\cal T} \leq p$,
 there is a canonical presheaf inclusion $z(X,r)_{ \mathcal T', p'} \subseteq z(X,r)_{ \mathcal T, p}$.
We say an incidence datum $(\cal T, p)$ is \textit{finite} if $\cal{T}$ consists of finitely many elements.

If $X$ is stratified and $\ol{p}$ is a perversity, we denote by $z_{equi}(X,r)_{\ol p}$ and 
$z(X,r)_{\ol p}$ the subpresheaves of $z(X, r)$ consisting of cycles whose excess intersection 
with (each component of) $X^i$ is bounded by $p_i$ (for all $i$).  
In other words, if $\cal T$ is the set of irreducible components of strata of the stratified variety, then $z(X,r)_{\ol p} = z(X, r)_{\cal T, p}$ for the function $p$ taking value $p_i$ on every component of $X^i$ (for every $i$).  
Put differently, $\alpha \in z(X, r)(U)$ 
belongs to $z(X,r)_{\ol p}(U)$ if for all $u \in U$, the specialization $\alpha_u$ belongs 
to $Z_{r, \ol p}(X_{u})$.

\begin{lemma} The cdh-sheafification of $z_{equi}(X, r)_{\mathcal T, p}$ is $z(X,r)_{\mathcal T, p}$.  
Therefore,  
$$   z(X,r)_{\ol p} \ \cong \ (z_{equi}(X,r)_{\ol p})_{cdh}.$$
\end{lemma}

\begin{proof} The cdh-sheafification of $z_{equi} (X ,r)$ is  $z (X ,r)$ \cite[Thm.~4.2.9]{SV}, so any cycle $\alpha \in  z(X,r)_{\mathcal T, p}(U)$ belongs to $z_{equi}(X,r)(U')$ for some cdh cover $p : U' \to U$.  Since the support of 
$\alpha_{u'}$ coincides with that of $\alpha_{p(u')}$,
in fact the base change of $\alpha$ lies in $z_{equi} (X, r)_{\mathcal T, p}(U')$.  \end{proof}

We prove two elementary functoriality properties for $X \mapsto z(X,r)_{\ol p}$.
We remark that proper push-forward is defined only under restrictive conditions;  
since
disjoint closed sets in the source of a morphism may have images which intersect,
the push-forward of a perversity cycle via a stratified morphism need not satisfy the same perversity condition.
In the following proposition, the perversity $\ol p * \ol c$ captures the interaction between the perversity of the cycle and the perversity describing the behavior of the stratification under the morphism.

\begin{prop}
Let $f: W \to X$ be a flat, stratified morphism of relative dimension $e$.  
Then for any perversity $\ol p$ and any $r \geq 0$, $f$ induces maps of {(pre)sheaves} 
$$f^*: z(X,r)_{\ol p} \to z(W,r+e)_{\ol p}, \quad f^*: z_{equi}(X,r)_{\ol p} \to z_{equi}(W,r+e)_{\ol p}.$$

If $f: W \to X$ is a proper morphism with the property that $W^{i-c_i} = f^{-1}(X^i)$
for some perversity $\ol c$, then for any 
perversity $\ol p$ and any $r \geq 0$, $f$ induces maps of {(pre)sheaves}
$$f_*:  z(W,r)_{\ol p} \to z(X,r)_{\ol p * \ol c}, \quad f_*: z_{equi}(W,r)_{\ol p} \to z_{equi}(X,r)_{\ol p * \ol c},$$
where $\ol p * \ol c$ is the perversity with ${(\ol p * \ol c)}_i = p_{i-c_i} + c_i$.

{In particular, if $i: W \to X$ is a closed immersion and $i(W)$ meets each $X^i$ properly,
 then such proper push-forward maps exist for $i$ if we take
each $c_i $ equal to 0.}

The pull-back and push-forward operations are compatible.
\end{prop}

\begin{proof} The existence statements for the presheaves with no perversity condition are \cite[Lemma 3.6.4]{SV} (flat pull-back) and \cite[Cor.~3.6.3]{SV} (proper push-forward).
Therefore the first assertion follows from the observation that
for any locally closed subset $T \subset X$, any flat map $f: W \to X$, and any $r$-cycle $\beta$ on $X$,
we have that $|f^*(\beta)| \cap f^{-1}(T) \ = \ f^{-1}(|\beta|\cap T).$  
The second assertion follows from the observation that, for any $r$-cycle $\alpha$ on $W$, we have $\dim ( |f_* (\alpha)| \cap X^i ) \leq r -i +c_i + p_{i-c_i}$  
since $f ( |\alpha| \cap W^{i- c_i}) = |f_*(\alpha) | \cap X^i$.

The flat pull-back and proper push-forward transformations are compatible \cite[Prop.~3.6.5]{SV}.
\end{proof}

\vskip .2in

\textbf{Motivic homology theories.}
The algebraic $n$-simplex is the affine variety $\spec (k[x_0,\ldots,x_n]/\sum_i x_i - 1)$ and is denoted by $\Delta^n$.  The schemes $\Delta^n$ fit together into a cosimplicial scheme $\Deltadot$.  
If $\cF$ is an abelian presheaf on $Sm / k$, we denote by $\cF(\bullet)$ the simplicial abelian group obtained by evaluation at $\Deltadot$.  For example, $z(X,r)_{\ol p}(\bullet)$ denotes the simplicial abelian group whose abelian group of $n$-simplices is $z(X,r)_{\ol p}(\Delta^n)$.
We denote by $C_*(\cF)$ (the ``Suslin complex" of $\cF$) the normalized chain complex of $\cF(\bu)$;
thus, $\pi_i(\cF(\bu)) = H_i(C_*(\cF))$.

{For $n \in \bb{Z}, r \in \bb{Z}_{\geq 0}$, the Borel-Moore motivic homology $H_n^{BM}(X, \bb{Z}(r))$ of $X \in Sch / k$ is the homology in degree $n-2r$ of the complex $C_* (z(X,r))$; for $r < 0$, $H_n^{BM}(X, \bb{Z}(r))$ is the homology in degree $n-2r$ of $C_* (z(X \times \bb{A}^{-r}, 0))$ \cite[4.3, 9.1]{FV}.  This motivates the following definition.
\begin{defn} \label{defn:pervhom}
The perversity $\ol p$ (Borel-Moore) motivic homology of a stratified variety $X$, written $H_n^{\ol p}(X, \bb{Z}(r))$, is the homology in degree $n-2r$ of the complex $C_* (z(X,r))_{\ol p}$. Equivalently, $H_n^{\ol p}(X, \bb{Z}(r)) \equiv \pi_{n-2r} (z(X,r)_{\ol p}(\bu))$.  
\end{defn}
The group $H_{2r}^{BM}(X, \bb{Z}(r))$ is the Chow group $A_r(X)$ of $r$-dimensional cycles on $X$. The group $H_{2r}^{\ol p}(X, \bb{Z}(r))$ admits a similar description.}

\begin{prop}
\label{Arp}
Consider $W_0, \ W_1 \in Z_{r, \ol p}(X) = z(X,r)_{\ol p}(k)$.   The following are equivalent:
\begin{enumerate}
\item
$W_0, \ W_1$ determine the same element in $\pi_0(z(X,r)_{\ol p}(\bu)) = H^{\ol p}_{2r}(X, \bZ(r))$.
\item 
$W_0, \ W_1$ determine the same element in $\pi_0(z_{equi}(X,r)_{\ol p}(\bu))$.
\item
There exists an $(r+1)$-dimensional cycle $\cW \into X \times \bb{A}^1$ satisfying the following properties:

(i.) $\cW$ is flat over $\bb A^1$;

(ii.) for all $t \in \bb{A}^1$, $\cW_t \in Z_{r, \ol p}(X_t)$; and

(iii.) $W_0 = \cW \bu (X\times 0)$ and $W_1 = \cW \bu (X \times 1) $.
\item
There exists an \textit{effective} 
$(r+1)$-dimensional cycle $\cW \into X \times \bb{A}^1$ satisfying (i.) and (ii.), 
and a cycle $E \in Z_{r, \ol p}(X)$ such that  $\cW \bu (X\times 0) =W_0 + E$ and $\cW \bu (X \times 1) = W_1+ E$. 
\end{enumerate}
If $W_0, \ W_1$ satisfy these conditions, then we say that they are rationally equivalent as $r$-cycles of 
perversity $\ol p$, written $W_0 \sim_{\ol p} W_1$.  We denote by $A_{r, \ol p}(X)$ the 
quotient of $Z_{r, \ol p}(X)$ by the relation $\sim_{\ol p}$:
\begin{equation}
\label{pervChow}
A_{r, \ol p}(X) \ \equiv \ Z_{r, \ol p}(X)/\sim_{\ol p}.
\end{equation}
\end{prop}

\begin{proof}
The equivalence of (1) and (2)  follows from the observation
 that relative cycles are automatically flat (hence equidimensional) over a smooth base of dimension $\leq 1$.
 
To show the equivalence of the second and third conditions, observe that elements of $z_{equi}(X,r)_{\ol p}(\Delta^1)$ 
are in bijective correspondence with $(r+1)$-dimensional cycles $\cW \into X \times \bb{A}^1$ satisfying 
 the conditions (i.) and (ii.) of the third condition.  
 
The equivalence of the third and fourth conditions is essentially verified in  \cite[Ex.~1.6.2]{Ful}.
\end{proof}

Forgetting the stratification of $X$ determines a group homomorphism from $A_{r, \ol p}(X)$ to rational equivalence
classes of $r$-cycles on $X$, 
$A_{r, \ol p}(X) \to A_r(X)$,  which need not be injective or surjective.

\vskip .1in

The following proposition establishes a perverse cycle class map from our perversity $\ol p$ Chow group to the Goresky-MacPherson group.  We {ignore} a slight notational conflict; our $p_i$ corresponds to $p_{2i}$ in the Goresky-MacPherson convention.  We use the geometric model for intersection homology as developed in \cite[1.3]{GM1}: instead of considering the usual complex of (locally finite) chains, one considers the subcomplex of chains whose excess intersection with the strata is controlled by $\ol p$, and with boundary satisfying a similar condition.  The homology groups of this complex are the intersection homology groups of perversity $\ol p$; these turn out to be independent of the stratification, as established via the sheaf-theoretic approach in \cite[\S 4, Cor.~1]{GM2}. 

Our original hope was to define purely algebro-geometrically a pairing $A_{r, \ol p} (X) \times A_{s, \ol q}(X) \to A_{r+s-d, \ol p + \ol q}(X)$ which agrees with the Goresky-MacPherson pairing via the perverse cycle class map.  The construction of such a pairing, {and the study of the dependence of our groups on the stratification,} seem to require additional geometric input.

\begin{prop}
\label{prop:GM} Let $X$ be a stratified variety of dimension $d$ over $\bb{C}$, and suppose the stratification is sufficiently fine to compute the intersection homology groups $IH^{\ol p}_*(X)$.  Let $A_{r, \ol p}(X)$ {(\ref{pervChow})} denote the perversity $\ol p$ Chow group with respect to the same stratification.  Then there is a {canonical} perverse cycle class map
$$A_{r, \ol p}(X) \to IH^{\ol p}_{2r}(X, \Z).$$
\end{prop}
\begin{proof} If $\alpha$ is an algebraic cycle in $Z_{r, \ol p}(X)$, then a triangulation of $\alpha$ determines a cycle in the intersection chain complex.  It suffices to show that if $\alpha \sim_{\ol p} \alpha'$, then the difference $\alpha - \alpha'$ goes to zero in $IH^{\ol p}_{2r}(X, \Z)$.
If $\alpha \sim_{\ol p} \alpha'$, then there exists a cycle $\mathcal{W}$ on $X \times \PP^1$ such that $\mathcal{W}_0 = \alpha + E$ and $\mathcal{W}_1 = \alpha' + E$, with $\alpha, \alpha', E \in Z_{r, \ol p}(X)$.

We equip $X \times \PP^1$ with the  stratification given by pulling back the stratification of $X$.  
We claim $\mathcal{W}$ determines a class in $IH^{\ol p }_{2r+2}(X \times \PP^1, \Z)$.  
This follows from the observation that if $Y \into X$ is a Cartier divisor, $\beta$ is an $(r+1)$-dimensional 
cycle on $X$ not contained in $Y$, $T \into X$ is closed, and the $r$-cycle $\beta \cap Y$ has 
excess $\leq e$ with $T \cap Y \into Y$ in $Y$, then $\beta$ itself has excess $\leq e$ with $T \into X$.

We utilize the intersection pairing 
$$H^2(X \times \PP^1, \Z) \times IH^{\ol p }_{2r+2}(X \times \PP^1, \Z) \to IH^{\ol p}_{2r}(X \times \PP^1, \Z) .$$
The pair $(X \times 0, \cal{W})$ intersects properly in each stratum $X^i \times \PP^1$ since $X^i \times 0$ does not contain $\cal{W} \cap (X^i \times \PP^1)$, and the same holds for the pair $(X \times \infty, \cal{W})$.  Therefore the product $[X \times 0 ] \cdot [\cal W]$ is represented by the class of $\cal{W}_0$, and similarly $[X \times \infty ] \cdot [\cal W] = [\cal{W}_\infty]$.  The divisors $X \times 0, X \times \infty \into X \times \PP^1$ determine the same class in $H^2(X \times \PP^1, \Z)$, so $[\cal{W}_0] = [\alpha + E] = [\alpha' + E] = [\cal{W}_\infty] \in IH^{\ol p}_{2r}(X \times \PP^1, \Z)$, hence $[\alpha] - [\alpha'] = 0 \in IH^{\ol p}_{2r}(X \times \PP^1, \Z)$.

There are push-forward morphisms $0_*, \infty_* : IH^{ \ol p}_{2r}(X, \Z) \to IH^{\ol p}_{2r}(X \times \PP^1, \Z)$ and a projection morphism $p_* : IH^{\ol p}_{2r}(X \times \PP^1, \Z) \to IH^{\ol p}_{2r}(X, \Z)$ \cite[Proof of Prop.~2.1]{GFried}.  Both $0$ and $\infty$ are sections to $p$, so  $p_* \circ 0_*$ and $p_* \circ \infty_*$ are both the identity, and this completes the proof. \end{proof}

{\begin{remark} For $X$ projective, C.~Flannery constructed a morphism from the homotopy groups of the space of algebraic cycles of some perversity (i.e., semi-topological intersection homology groups), to the Goresky-MacPherson groups \cite{Flannery}. \end{remark}}

\vskip .2in

\textbf{Applications of Voevodsky acyclicity.}
We now use Voevodsky's results {on the cohomology of }pretheories to relate $z(X,r)_{\ol p}(\bu)$ and $z_{equi}(X,r)_{\ol p}(\bu)$.  {A pretheory is a presheaf equipped with push-forward maps along relative divisors in relative smooth curves (over smooth bases).  }Here we show the subpresheaves defined by incidence data are in fact subpretheories.

\begin{lemma} \label{pretheory} Let $X$ be an equidimensional $k$-scheme, and let $(\cal T, p)$ be an incidence datum on $X$.
The subpresheaves $z_{equi}(X,r)_{\mathcal T, p}$ 
and $z(X,r)_{ \mathcal T, p}$ are subpretheories inside $z_{equi}(X,r)$ and $z(X,r)$.
\end{lemma}

\begin{proof} We recall that both $z_{equi}(X,r)$ and $z_{equi}(X,r)_{cdh} = z(X,r)$ admit canonical %pretheory
structures of pretheories 
(in the sense of Voevodsky) in such a way that the canonical morphism $z_{equi}(X,r) \to z_{equi}(X,r)_{cdh}$ 
is a morphism of pretheories \cite[Remark 5.10]{FV}.  
The pretheory structure is defined using intersection followed by push-forward along a finite morphism \cite[Prop.~5.7]{FV}.
For notational simplicity we treat here only the case $z(X,r)$.
So suppose $U$ is a smooth $k$-scheme, $C \to U$ is a smooth curve, and $Z \in c_{equi}(C/U,0)$ with 
morphisms $f: Z \to C$ and $p :Z \to U$.  (We use $c$ instead of $z$ to indicate the support of $Z$ is proper over $U$; see \cite[Lemma 3.3.9]{SV}.)
For $W \in z(X,r)(C)$, we first form the intersection $W_Z$ of $W \into C \times X$ with the Cartier divisor $Z \times X \into C \times X$.
The cycle $\phi_{C/U}(Z)(W) \in z(X,r)(U)$ is then the push-forward ${(p \times \id)}_* (W_Z)$ of $W_Z$ along (the proper morphism)
$p \times \id: Z \times X \to U \times X$.
In particular, the support of $\phi_{C/U}(Z)(W)$ 
at $u \in U$ is contained in 
the union of the supports of $W_c$ for $c \in f( p^{-1} (u))$.  
Therefore $W \in z(X,r)_{\cal T, p}(C)$ implies $\phi_{C/U}(Z)(W) \in z(X,r)_{\cal T,p}(U)$, and the subpresheaves $z_{equi}(X,r)_{\mathcal T, p}$ 
and $z(X,r)_{ \mathcal T, p}$ are subpretheories inside $z_{equi}(X,r)$ and $z(X,r)$. 
\end{proof}

{\begin{remark} The pretheory structure may be phrased as a coherent system of morphisms $c_{equi}(C/U, 0) \to \Hom (z(X,r)(C), z(X,r)(U))$ for all relative curves $C \to U$. 
 A presheaf with transfers $\cal{F}$ has push-forwards along all $Z \in c(U \times Y /U, 0)$ for $U,Y \in Sm /k$, i.e., is equipped with a coherent system of morphisms
$c(U \times Y /U, 0) \to \Hom (\cal{F}(Y), \cal{F} (U))$, hence has more structure than a pretheory \cite[Prop.~3.1.11]{TRICAT}. 
Since the construction in the proof of Lemma $\ref{pretheory}$ works with $Z \in c ( U \times Y / U, 0)$% if $U, Y \in Sm / k$
, the presheaves $z_{equi}(X,r)_{\ol p}$ and $z(X,r)_{\ol p}$ are in fact presheaves with transfers. 
\end{remark}  }

\begin{proposition} \label{vanishing}  Let $X$ be a quasi-projective variety, and let $(\cal T, p)$ be an incidence datum on $X$.
Assume $k$ admits resolution of singularities.  Then the canonical morphism 
$z_{equi}(X,r)_{\mathcal T, p} \to z(X,r)_{ \mathcal T, p}$ induces a quasi-isomorphism of Suslin complexes. 
\end{proposition}

\begin{proof} Since the canonical morphism of pretheories $z_{equi}(X,r)_{\mathcal T, p} \to z(X,r)_{ \mathcal T, p}$ 
becomes an isomorphism after cdh-sheafification, Voevodsky's results on the cohomology of pretheories
imply $C_* (z_{equi}(X,r)_{\mathcal T, p}) \to C_*(z(X,r)_{ \mathcal T, p})$ is a quasi-isomorphism 
\cite[Thm.~5.5(2)]{FV}.  \end{proof}

The key additional property satisfied by $z(X,r)$ and not $z_{equi}(X,r)$ is the following localization property.
By \cite[Thm.~4.3.1]{SV},  if $i : X_\infty \into X$ is a closed immersion with open complement $j: U \subset X$, 
there is an exact sequence of cdh-sheaves:
\begin{equation}
\label{exact-cdh}
0 \to z(X_\infty, r) \xrightarrow{i_*} z(X, r) \xrightarrow{j^*} z(U, r) \to 0.
\end{equation}
 There is {\it not} such a short exact sequence with $z(-)$ replaced by $z_{equi}(-)$.

\begin{thm} \label{cdh ses} Let $X$ be a quasi-projective variety, and let $(\cal T, p)$ be an incidence datum on $X$.
Suppose $j : X \subset \ol{X}$ is an open immersion with $\ol{X}$ projective, and 
let $i: X_\infty \into \ol{X}$ denote the closed complement.  {Assume $k$ admits resolution of singularities.}
There is an exact sequence of cdh-sheaves:
\begin{equation}
\label{short}
0 \to z(X_\infty, r) \xrightarrow{i_*} z(\ol{X}, r)_{\cal{T}, p} \xrightarrow{j^*} z(X, r)_{\cal{T}, p} \to 0
\end{equation} 
which determines a distinguished triangle of Suslin complexes 
$$C_* ( z(X_\infty, r)) \xrightarrow{i_*} C_* (z(\ol{X}, r)_{\cal{T}, p}) \xrightarrow{j^*} C_*(z(X, r)_{\cal{T}, p}) \to C_* ( z(X_\infty, r))[1]$$
and hence a long exact sequence of the corresponding homology groups. 
\end{thm}

\begin{proof} The exactness of the asserted exact sequence
 is clear except at the final term.  Given $\alpha \in z(X, r)_{\cal{T}, p}(U) \subset z(X, r)(U)$, 
by (\ref{exact-cdh}) there exists a cdh-cover $p : U' \to U$ and an element $\ol{\alpha'} \in z(\ol{X}, r)(U')$ such that 
$j^* (\ol{\alpha'}) = p^* (\alpha)$.  But  $p^* (\alpha) \in z(X, r)_{\cal{T}, p}(U')$, so by definition $\ol{\alpha'} \in z(\ol{X}, r)_{\cal{T}, p}(U')$.  
The distinguished triangle now follows from \cite[Thm.~5.5(2)]{FV}.
\end{proof}

\begin{remark} As alluded to towards the end of Section \ref{sec:one}, Proposition \ref{vanishing} and Theorem \ref{cdh ses} hold unconditionally with $\bZ[1/p]$ coefficients over a perfect field of positive characteristic $p$ by the result of Kelly \cite[Thm.~5.3.1]{Kelly}. \end{remark}

%%%%%%%%%%%%%%%%%%%%%%%%%%%%%%%%%%
%%%%%%%%%%%%%%%%%%%%%%%%%%%%%%%%%%%%

\section{Suspension  theorems}
\label{sec:three}

In this section, we adapt the proof of Lawson \cite{Lawson} (formulated in more algebraic terms in
\cite{F} and adapted further in \cite{FV}) to establish ``Lawson suspension theorems"
({$\bb{A}^1$-invariance})
for perversity cycles.  P.~Gajer \cite{Gajer} used similar ideas in the semi-topological setting.

Let $X$ be a projective variety of dimension $d$ equipped with an embedding $X \into \bb{P}^N$.  
There is an induced embedding $\Sigma(X) \into \Sigma(\bb{P}^N) = \bb{P}^{N+1}$, where 
$\Sigma(-)$ denotes the algebraic suspension.  If $\bb{P}^N \into \Sigma(\bb{P}^N)$ is 
defined by the vanishing of the suspension coordinate, then the identification 
$X = \Sigma(X) \cap \bb{P}^N$ allows us to view subvarieties of $X$ as subvarieties of $\bb{P}^{N+1}$.
If $X' \subset X$ is an open subscheme of a projective variety $X$ with complement $X_\infty$, then we define $\Sigma(X') \equiv \Sigma(X) - \Sigma(X_\infty)$; this is an open subscheme of $\Sigma(X)$.

If $(\cal T, p)$ is an incidence datum on $X$, then we define $\Sigma(\cal{T}) := \{ \Sigma(T) \}_{T \in \cal T}$
and we consider both $(\cal{T}, p)$ and $(\Sigma( \cal{T}), p)$ as incidence data on $\Sigma(X)$.

Our arguments in this section use geometric constructions on the Chow monoids and therefore 
our results concern presheaves of equidimensional cycles and their cdh-sheafifications.
To obtain the results for $z_{equi}(X, r)_{\cal{T}, p}$, we apply the functor
$\Hom {( - , \cC_{r+1}(\Sigma(X)))}^+$
(or the subfunctor of morphisms satisfying the field of definition condition)
to our constructions.
For $z(X,r)_{\cal T, p}$, we apply the functor $\Hom_{\contalg} ( - , \cZ_{r+1}(\Sigma(X)))$ or its field of definition subfunctor.
For both cases we observe our constructions respect the field of definition condition.

\begin{thm} \label{thm:proj-susp} Let $X$ be a projective variety, and let $(\cal T, p)$ 
be a finite incidence datum on $X$.  The fiberwise suspension morphism of presheaves
$$\Sigma_X: z_{equi}(X, r)_{\cal{T}, p} \to z_{equi}(\Sigma(X), r+1)_{\Sigma(\cal{T}), p}$$
sending an effective cycle $W \subset U \times X$ to the effective cycle $\Sigma_X(W) \subset U \times \Sigma (X)$
induces a homotopy equivalence
\begin{equation}
\label{basic}
z_{equi}(X, r)_{\cal{T}, p}(\bullet) \ \stackrel{\sim}{\to} \  z_{equi}(\Sigma(X), r+1)_{\Sigma(\cal{T}), p}(\bullet). 
\end{equation}
{The fiberwise suspension also induces a homotopy equivalence
$$z(X, r)_{\cal{T}, p}(\bullet) \ \stackrel{\sim}{\to} \  z(\Sigma(X), r+1)_{\Sigma(\cal{T}), p}(\bullet) .$$}
\end{thm}

We establish this homotopy equivalence (\ref{basic})  by factoring 
$\Sigma_X$ as a composition
\begin{equation} 
\label{factor} 
z_{equi}(X, r)_{\cal{T}, p} \to z_{equi}(\Sigma(X), r+1)_{X, \cal{T}, p}
 \to z_{equi}(\Sigma(X), r+1)_{\Sigma(\cal{T}), p} ,
 \end{equation}
 showing in Proposition \ref{prop:iso1} (respectively, in Proposition \ref{prop:iso2}) that the first (resp., second) 
morphism induces a homotopy equivalence upon evaluation at $\Deltadot$. 
 The presheaf $z_{equi}(\Sigma(X), r+1)_{X, \cal{T}, p}$
  consists of cycles meeting $X$ properly, and having excess intersection 
 with $T$ no larger than $p(T)$.  
 
 Before giving the proof of Theorem \ref{thm:proj-susp}, we state explicitly the special case of primary interest,
 the suspension isomorphism for perversity cycles on a stratified projective variety.

\begin{cor} 
\label{cor:susp} 
Let $X$ be a stratified projective variety, and let $\ol p$ be a perversity.  
Equip $\Sigma(X)$ with the stratification $\{ \Sigma(X^i) \}$, where $\{ X^i \}$ is the given stratification of $X$.
Fiberwise  suspension  induces homotopy equivalences

$$\Sigma_X: z_{equi}(X, r)_{\ol p}(\bullet) \ \stackrel{\sim}{\to} \ z_{equi}(\Sigma(X), r+1)_{\ol p}(\bullet) \  \text{and}$$

{$$\Sigma_X: z_{}(X, r)_{\ol p}(\bullet)  \ \stackrel{\sim}{\to} \  z_{}(\Sigma(X), r+1)_{\ol p}(\bullet).$$}
\end{cor}

The proof of our first homotopy equivalence uses the technique of deformation to the normal 
cone (see \cite[Ch.~5]{Ful}), called ``holomorphic taffy" by Lawson in \cite{Lawson}.

\begin{prop} 
\label{prop:iso1}
Retain the notation and hypotheses of Theorem \ref{thm:proj-susp}. The morphism 
$\Sigma_X: z_{equi}(X, r)_{\cal{T}, p} \to z_{equi}(\Sigma(X), r+1)_{X, \cal{T}, p}$ 
induces a homotopy equivalence $z_{equi}(X, r)_{\cal{T}, p}(\bullet)  \ \stackrel{\sim}{\to} \ 
z_{equi}(\Sigma(X), r+1)_{X, \cal{T}, p}(\bullet)$.  The same result holds for the cdh-sheafification.
 \end{prop}

\begin{proof}
Let $\cC_{r+1,d}(\Sigma(X))_X$ denote the open subset of the Chow variety consisting of cycles 
$\alpha$ such that $\alpha \cap X$ has dimension $r$, i.e., $\alpha$ is not contained in $X$.  
The suspension morphism $\Sigma_X : \cC_{r,d}(X) \to \cC_{r+1,d}(\Sigma(X))$ factors through 
$\cC_{r+1,d}(\Sigma(X))_X$.  As shown in \cite[Prop.~3.2]{F}, there is a continuous algebraic map (i.e., a morphism on
semi-normalizations)
$$\varphi: \cC_{r+1,d}(\Sigma(X))_X \times \bb{A}^1 \to \cC_{r+1, d}(\Sigma(X))_X$$
with the following properties \cite[Prop.~3.2]{F}.  (Here $\varphi_t$ denotes the restriction of
$\varphi$ to $\cC_{r+1,d}(\Sigma(X))_X \times \{ t \}$.)

\begin{enumerate}
\item $\varphi_0$ is the identity on $\cC_{r+1, d}(\Sigma(X))_{X}$;
\item $\varphi_1$ has image contained in $\Sigma_X (\cC_{r,d}(X))$, in fact 
$\varphi_1(\alpha) = \Sigma_X (\alpha \cap X)$; and
\item $\varphi_t$ acts as the identity on $\Sigma_X (\cC_{r, d}(X))$ for all $t \in \bb{A}^1$, 
in fact $\varphi_t$ (for $t \neq 1$) is induced by an automorphism of $\bb{P}^{N+1}$ 
fixing the suspension hyperplane $\bb{P}^N$.
\item $\varphi$ does not depend on the degree $d$.
\end{enumerate}

From properties (2) and (3) it follows that $\varphi$ preserves the field of definition of a cycle.  For $t  \neq 1$, we use that the automorphism is defined over the ground field $k$.  For $t=1$, the operation $\alpha \mapsto \Sigma_X (\alpha \cap X)$ may be described as eliminating all instances of the suspension coordinate in the equations defining $\alpha$.  

We adapt this construction as follows.
For any $U {\in Sm /k}$, let $ z_{equi}(\Sigma(X), r+1)_{X, \cal{T}, p}(U) \ \subset \ z_{equi}(\Sigma(X), r+1)(U)$ 
consist of those $U$-relative
cycles with the property that each specialization meets $X$ properly 
and meets $T \in \cal{T}$ with excess at most $p(T)$.
We proceed to show that $\varphi$ induces a morphism of presheaves (on $Sm / k$): 
\begin{equation}
\label{eq:phi}
\varphi_{\cal T, p} : z_{equi}(\Sigma(X), r+1)_{X, \cal{T}, p} (-) \to z_{equi}(\Sigma(X), r+1)_{X, \cal{T}, p} (- \times \bb{A}^1)
\end{equation}
with the following properties:
\begin{enumerate}
\item $(\varphi_{\cal T, p})_0$ is the identity on $ z_{equi}(\Sigma(X), r+1)_{X, \cal{T}, p}$;
\item $(\varphi_{\cal T, p})_1$ has image contained in $\Sigma_X( z_{equi}(X, r)_{\cal{T}, p})$,
 in fact $(\varphi_{\cal T, p})_1(\alpha) = \Sigma_X (\alpha \cap X)$ for any 
 $\alpha \in  z_{equi}(\Sigma(X), r+1)_{X, \cal{T}, p}(U)$  ; and
\item $(\varphi_{\cal T, p})_t$ acts as the identity on $\Sigma_X( z_{equi}(X, r)_{\cal{T}, p})$ 
for all $t \in \bb{A}^1$, in fact $(\varphi_{\cal T, p})_t$ (for $t \neq 1$) is induced 
by an automorphism of $\bb{P}^{N+1}$ fixing the suspension hyperplane $\bb{P}^N$.
\end{enumerate}
\noindent (Here $(\varphi_{\cal T, p})_t$ denotes $\varphi_{\cal T, p}$ followed by restriction to $( - \times \{ t \})$.)
There is a canonical inclusion of presheaves of 
abelian monoids on $Sm/k$:
$$z^{eff}_{equi}(\Sigma(X), r+1)_{X, \cal{T}, p} (-) \to \Hom(-, \cC_{r+1}(\Sigma(X))_X )$$
which induces
$$z_{equi}(\Sigma(X), r+1)_{X, \cal{T}, p} (-) \to \Hom(-, \cC_{r+1}(\Sigma(X))_X )^+.$$

Now $\varphi$ induces a natural transformation
\begin{equation}
\label{eq:Phi}
\Hom(-, \cC_{r+1}(\Sigma(X))_X )^+   \to \Hom(-  \times \bb{A}^1, \cC_{r+1}(\Sigma(X))_X )^+
\end{equation}
{sending a morphism $f: U \to \cC_{r+1}(\Sigma(X))_X$ to the composition $\varphi \circ (f \times \id_{\bb{A}^1}) : U \times \bb{A}^1 \to \cC_{r+1}(\Sigma(X))_X$.}
We claim this restricts to our desired morphism $\varphi_{\cal T, p}$.
Properties (2) and (3) of $\varphi$ 
imply that for all $\alpha \in \cC_{r+1}(\Sigma(X))_X$ and all $t \in \bb{A}^1$, we 
have $\varphi_t (\alpha) \cap T = \alpha \cap T,$ so the incidence 
conditions with the sets appearing in $\cal{T}$ are preserved.
We have already observed that $\varphi$ 
preserves the field of definition of a cycle.

We are now in a position to apply \cite[Lemma 6.6]{FV}, and this completes the proof for the $equi$-theory.
If we work with {continuous algebraic maps into $\cZ_{r+1}(\Sigma(X))$} instead of
$\Hom(-, \cC_{r+1}(\Sigma(X))_X )^+$, we obtain the result for the (non-equidimensional) cdh theory $z(-,-)$.
\end{proof}

The proof of our second homotopy equivalence uses the technique first
introduced by Lawson in \cite{Lawson}, which he calls ``magic fans."

\begin{prop} 
\label{prop:iso2}
Retain the notation and hypotheses of Theorem \ref{thm:proj-susp}.   The canonical inclusion
$z_{equi}(\Sigma(X), r+1)_{X, \cal{T}, p} \to z_{equi}(\Sigma(X), r+1)_{\Sigma(\cal{T}), p}$ 
induces a homotopy equivalence 
$$z_{equi}(\Sigma(X), r+1)_{X, \cal{T}, p}(\bullet)  \ \stackrel{\sim}{\to} \  z_{equi}(\Sigma(X), r+1)_{\Sigma(\cal{T}), p}(\bullet).$$
The same result holds for the cdh-sheafification.
\end{prop}

\begin{proof} 
Let $\alpha \in Z_{r, \ol{p}}(X)$, 
and suppose $\alpha = \alpha^+ - \alpha^-$, where $\alpha^+$ and $\alpha^-$ are effective cycles 
with no components in common.  Then
 $Z_{r,\leq d, \ol{p}}(X) \subset Z_{r, \ol{p}}(X)$ 
consists of those cycles $\alpha$ such that $ \deg (\alpha^+) \leq d$ and $\deg (\alpha^-) \leq d$
(with respect to the given closed embedding $X \subset \bP^N$).  {Since the degree is
invariant under field extensions, this pointwise condition defines
a subpresheaf in our cycle presheaves.}

As shown in \cite[Prop.~3.5]{F}, for every $d \geq 0$, there exists an integer $e_d$ such that for 
every $e \geq e_d$ there exists a morphism of semi-normal schemes
\begin{equation}
\label{eq:psi}
\psi_e : \cC_{r+1, \leq d} (\Sigma(X)) \times \bb{A}^1 \to \cC_{r+1, \leq de} (\Sigma(X))
\end{equation}
with the following properties:
\begin{enumerate}
\item $\psi_e (\alpha, 0) = e \cdot \alpha$ for all $\alpha \in \cC_{r+1, \leq d} (\Sigma(X))$; and
\item $\psi_e (\alpha, t) \in \cC_{r+1, \leq de}(\Sigma(X))_X$, for all $\alpha \in \cC_{r+1, \leq d} (\Sigma(X))$ and all $t \neq 0  \in \bb{A}^1$.
\end{enumerate}

{Since the $\bb{A}^1$ corresponds a family of divisors defined over the ground field, and the suspension and projection operations preserve the field of definition of a cycle, the morphism $\psi_e$ preserves the field of definition condition.}

For ease of exposition we introduce some notation.  
Let $\cal{F}'_{\leq d}$ denote the presheaf $z^{eff}_{equi}(\Sigma(X), r+1, \leq d)_{X, \cal{T}, p}$, 
and let $\cal{F}_{\leq d}$ denote the presheaf $z^{eff}_{equi}(\Sigma(X), r+1, \leq d)_{\Sigma(\cal{T}), p}$.  
We have the following commutative diagram
of canonical inclusions of presheaves on $Sm/k$:
$$\xymatrix{\cal{F}'_{\leq d} \ar[r] \ar[d] & \Hom(-, C_{r+1, \leq d}(\Sigma(X))_X) \ar[d] \\
\cal{F}_{\leq d} \ar[r] &  \Hom(-, C_{r+1, \leq d}(\Sigma(X)))} $$
We let $z_{equi}(\Sigma(X), r+1, \leq d)_{\Sigma(\cal{T}), p}$ denote the quotient of 
$\cal{F}_{\leq d} \times \cal{F}_{\leq d}$ by the evident relation: $(a,b) \sim (a', b')$ if $a+b' = a' +b$ 
as cycles.  Note that $z_{equi}(\Sigma(X), r+1)_{\Sigma(\cal{T}), p} = 
\cup_d z_{equi}(\Sigma(X), r+1, \leq d)_{\Sigma(\cal{T}), p}$.  We employ the analogous notation with the subscript $X, \cal{T}, p$.

We claim that $\psi_e$ of (\ref{eq:psi}) restricts to  a morphism of presheaves 
$(\psi_e)_{\cal T,p}: \cal{F}_{\leq d} (-)  \to  \cal{F}_{\leq de} (- \times \bb{A}^1)$ with the following properties:
\begin{enumerate}
\item ${((\psi_e)_{\cal T,p})}_0( \alpha) = e \cdot \alpha$ for all $\alpha \in \cal{F}_{\leq d}(U)$; and
\item ${((\psi_e)_{\cal T,p})}_t (\alpha) \in \cal{F}'_{de}(U)$ for all $\alpha \in \cal{F}_{\leq d}(U)$ and all $t \neq 0 \in \bb{A}^1$.
\end{enumerate}

{Since the operation $\psi_e$ affects only the suspension coordinate, it follows that }
$${( \psi_e )}_t (\alpha) \cap \Sigma(T) =  {(\psi_e)}_t  ( \alpha \cap \Sigma(T) )$$
{for all $\alpha \in \cC_{r+1}(\Sigma(X)), \ t \in \bA^1$. The right hand side is controlled by hypothesis, and a bound on the dimension of the left hand side defines $\cal{F}_{\leq de}$.  Therefore $\psi_e$ restricts to a morphism on the subpresheaf $\cal{F}_{\leq d}.$}

{The first property is immediate from the corresponding condition of $\psi_e$.  The second property means that $\psi_e$ improves the incidence with $T \into \Sigma(T)$ and with $X \into \Sigma(X)$.  The improvement with $X \into \Sigma(X)$ is due to \cite[Prop.~3.5]{F}, and the incidence with $T$ is handled similarly.  Namely, given a bounded family of cycles $\{ \alpha \}$ on $\Sigma(X)$ satisfying the $(\Sigma(\cal{T}), p)$ condition, we consider the bounded families of $(r+1- \codim_X(T) +p(T))$-dimensional cycles $\{ |\alpha \cap \Sigma(T) | \}$ for all $T \in \cal T$. Following \cite[Prop.~3.5]{F} we find a $\PP^1$-family of hypersurfaces (of large degree $e$ depending on these bounded families) through $e \cdot \PP^{N+1}$ such that no member (besides $e \cdot \PP^{N+1}$) contains any of the cycles in the bounded families. (The finiteness of $\cal T$ guarantees we can find a family which works uniformly.) This guarantees the moved cycle satisfies the (stronger) $(X, \cal T, p)$ condition.}

{The rest is formal.}  The morphism of presheaves
$$(\cal{F}_{\leq d} \times \cal{F}_{\leq d}) (-) \to z_{equi}(\Sigma(X), r+1)_{\Sigma(\cal{T}), p}(- \times \bb{A}^1)$$
defined by
$$(a,b) \mapsto ({(\psi_{e+1})}_{\cal T, p}(a) - {(\psi_{e})}_{\cal T, p}(a)) \ -  \ ( {(\psi_{e+1})}_{\cal T, p}(b) - {(\psi_{e})}_{\cal T, p}(b) )$$
determines a natural transformation
\begin{equation}
\label{psi_e}
z_{equi}(\Sigma(X), r+1, \leq d)_{\Sigma(\cal{T}), p}(-) \to z_{equi}(\Sigma(X), r+1)_{\Sigma(\cal{T}), p} (- \times \bb{A}^1)
\end{equation}
which relates the identity (at $t=0$) to a morphism factoring (for all $t \neq 0$) through 
$z_{equi}(\Sigma(X), r+1)_{X, \cal{T}, p}$.  
Now \cite[Lemma 6.6]{FV} completes the proof, as in the conclusion of the proof of Proposition \ref{prop:iso1}.
\end{proof}

We next extend Theorem \ref{thm:proj-susp} and Corollary \ref{cor:susp} to quasi-projective varieties.  The proof employs the 
localization theorem for  $z(X,r)_{\cal T,p}$ and the comparison of $z_{equi}(X,r)_{\cal T,p}$ with $z(X,r)_{\cal T,p}$,
and thus requires that $k$ admits 
resolution of singularities, or inverting $\charct(k) = p > 0$ in the coefficients.
Localization provides us with the distinguished triangles of Proposition $\ref{cdh ses}$
which we use to reduce the case of $X$ quasi-projective to the consideration of the projective closure $\overline X$
of $X$ and the projective complement $X_\infty = \overline X - X$.

\begin{thm}
\label{thm:quasi-susp} 
Let $X$ be a quasi-projective variety, and let $(\cal T, p)$ be a finite incidence datum on $X$.
Assume $k$ admits resolution of singularities.  Then the morphism of presheaves
$$\Sigma_X: z_{equi}(X, r)_{\cal{T}, p} \to z_{equi}(\Sigma(X), r+1)_{\Sigma(\cal{T}), p}$$
induces a homotopy equivalence
$$z_{equi}(X, r)_{\cal{T}, p}(\bullet) \stackrel{\sim}{\to} z_{equi}(\Sigma(X), r+1)_{\Sigma(\cal{T}), p}(\bullet).$$
Consequently, if $X$ is a stratified quasi-projective variety and  $\ol p$ is a perversity, then the
suspension morphism of presheaves induces homotopy equivalences
$$\Sigma_X: z_{equi}(X, r)_{\ol p}(\bullet)\ \stackrel{\sim}{\to} \  z_{equi}(\Sigma(X), r+1)_{\ol p}(\bullet)  \ \text{ and} $$
{$$\Sigma_X: z_{}(X, r)_{\ol p}(\bullet)\ \stackrel{\sim}{\to} \  z_{}(\Sigma(X), r+1)_{\ol p}(\bullet).$$}
\end{thm}
  
  \begin{proof} By Proposition $\ref{vanishing}$, it suffices to prove the statements for the morphism 
$\Sigma_X: z(X, r)_{\cal{T}, p} \to z(\Sigma(X), r+1)_{\Sigma(\cal{T}), p}$.  

Choose a projective compactification $\ol{X}$ of $X$, and regard $(\cal{T}, p)$ as an incidence datum on $\ol{X}$.
The morphism $\Sigma_{X_\infty} : z(X_\infty,r) \to z(\Sigma( X_\infty), r+1)$ induces a quasi-isomorphism of 
Suslin complexes by the usual $\bb{A}^1$-homotopy invariance \cite[Thm.~8.3(1)]{FV} and the isomorphism of 
sheaves $z(\Sigma(X_\infty), r+1) \cong z(X \times \bb{A}^1, r+1)$. 
The morphism $\Sigma_{\ol X} : z(\ol X, r)_{\cal{T}, p} \to z(\Sigma(\ol X), r+1)_{\Sigma(\cal{T}), p}$ 
induces a homotopy equivalence after evaluation at $\Deltadot$ by Theorem $\ref{thm:proj-susp}$.  

The suspension map determines a map of distinguished triangles of Suslin complexes 
as in Proposition \ref{cdh ses} which determines a map of long exact sequences of homology groups.
We view these homology groups as the homotopy groups of the simplicial abelian groups obtained
by applying  to $\Deltadot$ the short exact sequences of sheaves of the form (\ref{short}).  
Thus, the 5-Lemma enables us to
conclude the asserted isomorphisms.
 \end{proof}

\begin{remark} Theorem \ref{thm:quasi-susp} holds unconditionally over a perfect field of positive characteristic $p$ with $\bZ[1/p]$ coefficients. \end{remark}

%%%%%%%%%%%%%%%%%%%%%%%%%%%%%%%%

\vskip .2in
\section{Some intersection products}
\label{sec:four}

Let $X$ be a possibly singular variety of pure dimension $d$ with smooth locus $X^{sm}$ open in $X$ 
and singular locus $X_{sing} = X-X^{sm}$.  
{For the remainder of the paper we assume that $X$ is equipped with a stratification such that the singular locus of 
$X$ is contained in $X^1$.  }

Let $V, \ W$ be closed irreducible subvarieties of $X$
of dimension $r, \ s$ respectively and assume that the dimension of the intersection of the supports
$|V| \cap |W|$ is $\leq r+s-d$ (i.e., $V, \ W$ intersect properly).  Assume that no component of 
$|V| \cap |W| \cap X_{sing}$ has dimension $\geq r+s-d$.  Then we justify in Theorem \ref{thm:star}
our view that a good candidate for $V \bullet W$ on $X$ is the closure in $X$ of  the usual intersection
product of $V \cap X^{sm}, \ W \cap X^{sm}$ on the smooth variety $X^{sm}$.

With this in mind, we first formalize a stratified version of ``proper" intersection of cycles on a 
possibly singular variety $X$.

\begin{defn}
\label{def:star}
Let $X$ be a stratified variety of pure dimension $d$, let $\alpha, \beta$ be algebraic cycles on $X$ of dimension
$r,s$, and let $\ol c$ be a perversity.  Then $(\alpha, \beta)$ is said to satisfy condition $(*,\ol c)$ 
provided that 
        $$\dim(|\alpha| \cap |\beta| \cap X^i) \leq r+s-d-(i-c_i), \quad \text{for all }  i =1, \ldots, d$$
and $\dim(|\alpha| \cap |\beta| ) \leq r+s-d$.
\end{defn}

As we shall see in Section \ref{sec:cup-cap} (Propositions \ref{smooth pairing} and \ref{cap}, and Corollary \ref{sing pairing}), such pairs are provided by cycles of perversity $\ol p$ and generalized cocycles of perversity $\ol q$, if $\ol p + \ol q \leq \ol t$, where $\ol t$ denotes the top perversity.

\begin{thm}
\label{thm:star}
Let $X$ be a stratified variety of pure dimension $d$.
Let $z_{r*s,\ol c}(X) \ \subset \ z(X,r) \times z(X,s)$ denote the subsheaf on $(Sch/k)$
consisting of pairs
satisfying condition $(*,\ol c)$.  Then the {closure of the} intersection pairing on the smooth
locus of $X$ {defines} 
a morphism of functors on $(Sch / k)$: 
$$\bu: z_{r*s, \ol c}(X) \ \to \ z(X, r+s-d)_{\ol c}.$$
\end{thm}

\begin{proof} A pair $(\alpha, \beta) \in z(X,r)(U) \times z(X,s)(U)$ belongs to $z_{r*s,\ol c}(X)(U)$ provided every specialization $(\alpha_u, \beta_u)$ satisfies $(*, \ol c)$ on $X_{u}$.  If $U' \to U$ is a morphism in $Sch / k$, then the specialization of $(\alpha, \beta)_{U'}$ at $u' \in U'$ has support equal to the base change via $u' \to u$ of the support of $|\alpha_u| \cap |\beta_u|$, hence satisfies $(*, \ol c)$.  Therefore the condition $(*, \ol c)$ defines a presheaf.

The morphism of functors is determined by the intersection product on the smooth locus of $X$.  For the moment assume $U$ is integral with generic point $\eta$.  We send $(\alpha, \beta)$ to $\alpha \bu \beta$, defined to be the closure in $X \times U$ of the $r+s-d$-dimensional cycle $(\alpha_\eta )^{sm} \bullet_{X^{sm}_\eta} (\beta_\eta)^{sm}$ in $X^{sm}_\eta$.  This is a cycle on $X \times U$ whose generic points lie over $\eta$, so we need to show it has well-defined specializations.
 
Every pair $(\alpha_u, \beta_u)$ satisfies $(*, \ol c)$, therefore $|\alpha_u| \cap |\beta_u|$ has its generic points in $X^{sm}_{u}$ for every $u \in U$.  The intersection product on smooth varieties is compatible with specialization, so the specialization of $\alpha \bu \beta$ along a fat point $(x_0, x_1)$ over $u \in U$ is the closure in $X_{u}$ of the intersection product of $((x_0, x_1)^* (\alpha))^{sm}$ and $((x_0,x_1)^* (\beta))^{sm}$ in ${X^{sm}_{u}}$.  By hypothesis the specializations of $\alpha$ and $\beta$ are independent of the choice of fat point, so the same is true of $\alpha \bu \beta$.  Since the intersection product preserves integral coefficients, if $\alpha$ and $\beta$ have universally integral coefficients then so must $\alpha \bu \beta$.

If $U$ has several irreducible components, we define $\alpha \bu \beta$ by the procedure above on each component.  
Where the components of $U$  meet, the specializations agree since they may be described in terms of 
specializations of $\alpha$ and $\beta$, which agree by hypothesis.  \end{proof}

\begin{remark} \label{homotopy pairing}
{We remind the reader that a $\bZ$-bilinear pairing $A_\bu \times B_\bu \to C_\bu$ of simplicial abelian groups 
factors as a map of simplicial sets through the smash product of $A_\bu$ and $B_\bu$, 
$$A_\bu \times B_\bu \ \to \ A_\bu \wedge  B_\bu \ \to \ C_\bu$$
and thus determines a pairing on homotopy groups}
\begin{equation}
\label{smash}
\pi_i(A_\bu) \otimes \pi_j(B_\bu) \ \to \pi_{i+j}(C_\bu).
\end{equation}

{Theorem \ref{thm:star} identifies a subsheaf of $z(X,r) \times z(X,s)$ on which intersections can be formed.  The maps $z(X,r) \times z(X,s) \leftarrow z_{r*s, \ol c}(X) \to z(X, r+s-d)_{\ol c}$ then determine a partially defined pairing on homotopy groups.}
\end{remark}

We make explicit the following special case of the functoriality of Theorem \ref{thm:star}.  In fact, much of the above proof of 
Theorem \ref{thm:star} can be interpreted as confirming the commutativity of the diagram in the following corollary.

\begin{cor}
\label{specialization}
Retain the notation and hypotheses of Theorem \ref{thm:star}.  Let $C$ be a smooth and connected curve,
let $\eta \in C$ be the generic point of $C$,
and let $\gamma \in C$ be a closed point of $C$.  
Then the following diagram commutes
$$\xymatrix{
z_{r*s, \ol c}(X)(\eta)  \ar[r]^-\bu & z(X, r+s-d)_{\ol c}(\eta) \\
z_{r*s, \ol c}(X)(C) \ar[d] \ar[u] \ar[r]^-\bu & z(X, r+s-d)_{\ol c}(C) \ar[u] \ar[d]\\
z_{r*s, \ol c}(X)(\gamma)  \ar[r]^-\bu & z(X, r+s-d)_{\ol c}(\gamma)
}$$
\end{cor}

\vskip .1in

\begin{prop}
\label{isolated sings}
Let $X$ be a projective variety of dimension $d$ with only isolated singularities and stratified by $X_{sing} = X^d = \cdots = X^1$.  Let $\ol p$ and $\ol q$ be perversities such that $\ol p + \ol q = \ol t$, and let $r$ and $s$ be positive integers such that $r+s-d \geq 0$.  Then the canonical inclusion $z_{r*s, \ol p + \ol q}(X) \to z(X,r)_{\ol p} \times z(X,s)_{\ol q}$ induces a homotopy equivalence

$$z_{r*s, \ol t}(X)(\bu) \stackrel{\sim}{\to}  (z(X,r)_{\ol p} \times z(X,s)_{\ol q} )(\bu)$$

and hence there is an intersection pairing

$$H^{\ol p}_{n}(X, \bb{Z}(r)) \otimes H^{\ol q}_{m}(X, \bb{Z}(s)) \to H^{}_{n+m-2d}(X, \bb{Z}(r+s-d)).$$
\end{prop}
\begin{proof}
First we consider the case that the perversities both permit intersection with the singular locus.  This means $r-d + p_d$ and $s-d +q_d$ are both non-negative, hence $r+s -d \geq d-(p_d + q_d)$.  Since $p_d + q_d \leq d-1$ this implies $r+s -d \geq 1$.  In this situation we use the functor isomorphisms $z(X - X_{sing}, r) \cong z(X,r) = z(X,r)_{\ol p}$ (and similarly with $r$ replaced by $s$ or $r+s-d$) and we are reduced to the case where $X$ is smooth.  Then the equivalence and pairing are consequences of the Friedlander-Lawson moving lemma for families (\cite[Thm.~3.1]{FLmoving}, as in \cite[Prop.~8.6]{FV}).

The hypothesis $\ol p + \ol q = \ol t$ implies at least one of the perversity conditions allows incidence with the singular locus; without loss of generality we suppose this is $\ol p$, and $\ol q$ disallows incidence with the singular locus.  Let $\alpha \in z(X, r)_{\ol p}(T)$ and $\beta \in z(X,s)_{\ol q}(T)$ be bounded families.  

Now \cite[Thm.~1.7]{FLmoving} implies we can find a sequence of projections $X \to \PP^d$ such that the iterated residual cycle of every $\alpha_t$ meets properly every $\beta_t$ except in the singular locus.  But every $\beta_t$ is disjoint from the singular locus, so every $\alpha_t$ meets properly every $\beta_t$.  Using moves in projective space one obtains the desired move parameterized by a non-empty open subset of $\bb{A}^1$.  To obtain an honest $\bb{A}^1$-family, we use the trick of Voevodsky (as in the proof \cite[Thm.~6.1]{FV}).  The existence of this $\bb{A}^1$-family implies the equivalence upon evaluation at $\Deltadot$ by the usual argument (\cite[Lemma 6.6]{FV}).
\end{proof}

\begin{remark} The iterated residual cycle construction does not seem adequate to move a bounded family of cycles $\alpha$ all of whose elements are disjoint from $X_{sing}$ into general position (with respect to another bounded family $\beta$) while preserving disjointness from $X_{sing}$.  Let $s \in X_{sing}$ be a singular point of $X$, and let $p_1(\alpha) \subset X$ denote the ``sweep" of $\alpha$, i.e., the image of the morphism $\alpha \into X \times T \to X$.  Let $Y \into \bb{G}(1,N)$ denote the set of lines connecting $s$ to a point in the sweep, i.e., $\{ \ell(s,z) | z \in p_1(\alpha) \}$. 
Now $\alpha \to X$ cannot be surjective since every $\alpha_t$ misses $s$, but it can be dense, and if it is dense then $\dim Y = d := \dim X$.

Any projection $p : X \to \PP^d$ arises from an embedding $X \into \PP^N$ and a choice of linear space $V \cong \PP^{N-d-1} \into \PP^n$, i.e., a point $[V] \in \bb{G}(N-d-1, N)$.  The residual cycle of $\alpha$ with respect to $p$ is disjoint from $s$ exactly when $p(s) \notin p(\alpha)$, which means the corresponding linear space $V$ must be disjoint from all lines in $Y$.

Now consider the incidence correspondence $\cal{I} \into \bb{G}(1,N) \times \bb{G}(N-d-1,N)$.  The codimension of $\cal{I}$ is $d$.  If $\dim Y = d$ the morphism $\cal{I} \cap (Y \times \bb{G}(N-d-1,N)) \to \bb{G}(N-d-1,N)$ may be surjective, i.e., the union of the codimension $d$ sets $\cal{I}_y \into \bb{G}(N-d-1,N)$ (as $y$ ranges over $Y$) could comprise all of $\bb{G}(N-d-1,N)$.  In this situation the projection required for the residual cycle construction does not exist.
\end{remark}

\begin{ex}
We consider a simple example due to Zobel \cite{Zobel} of a singular variety $X$ on which
there is no decent intersection product on usual rational equivalence classes of cycles.
Namely,  $X$ is the cone on a quadric surface $Q \into \PP^3$, i.e., on $\PP^1 \times \PP^1 \cong Q$.  
We refer to the unique singular point of $X$ as its vertex $v$.

We use the ``obvious" stratification, namely, $v = X^3  = X^2 = X^1 \into X$.  Since $p_3 \leq p_2 +1$ and $p_3 \leq p_1+2$, the condition on the incidence with $X^3$ determines the perversity.  Therefore we abuse notation and write $\ol p$ for any perversity with $p_3 = p$, where $p \in \{ 0,1,2 \}$.

By the $\bb{A}^1$-invariance of Chow groups, we have:
\begin{itemize}
\item $A_2(X) = A_{2,\ol 1}(X) = A_{2, \ol 2}(X) \cong \bb{Z} \oplus \bb{Z}$, 
with generators corresponding to cones on the two rulings of $\PP^1 \times \PP^1$; and
\item $A_1(X) = A_{1,\ol 2}(X) \cong \bb{Z}$, with generator corresponding to the cone on a point in $\PP^1 \times \PP^1$.
\end{itemize}

The classes of the lines $L = \PP^1 \times q, \ M= p \times \PP^1 \into Q \into X$ are equal in $A_1(X)$, and each generates.  
Note that each is rationally equivalent to $N = C(p \times q)$.  The lines $L$ and $M$ are contained in $X^{sm}$ but $N$ is not.
Consider the divisor $D = C(\PP^1 \times q')$ in $A_2(X)$ for some $q' \neq q$.  
We have $|D| \cap |L| = \emptyset$ while $|D| \cap |M| = p \times q' \in Q \into X$, and surely the coefficient 
of $p \times q'$ should be $1$.  Therefore, there is no reasonable pairing 
$A_2(X) \times A_1(X) \to A_0(X) \xrightarrow{\deg} \bb{Z}$, even if we consider only intersections 
which occur in the smooth locus of $X$.  Note that Proposition \ref{isolated sings} implies that any rational equivalence between $L$ and $M$ 
passes through the vertex, and that the classes of $L$ and $M$ must be distinct in $A_{1, \ol 0}(X)$.

We proceed to compute the intersection pairing (guaranteed by Proposition \ref{isolated sings}) on the intersection Chow groups. 

To calculate the zero perversity groups, we use that $X$ is birational to $\PP^1 \times Q$, and that geometry away 
from the vertex corresponds to geometry away from $\infty \times Q$.
Taking the birational transform of divisors and rational equivalences (all missing the vertex) identifies 
$A_{2,\ol 0}(X)$ with the relative Picard group $\pic(\PP^1 \times Q, \infty \times Q)$.  
We have $\pic(\PP^1 \times Q, \infty \times Q) \cong \bb{Z}$ (generated by  $\cal{O}(1)$ of the fiber of $\PP^1 \times Q \to Q$) 
since line bundles pulled back from $Q$ have nontrivial intersections with the divisor $\infty \times Q$.  
In essence we use the exact sequence
$$\Gamma(\PP^1 \times Q, \cal{O}^*) \to \Gamma(\infty \times Q, \cal{O}^*) \to 
\pic(\PP^1 \times Q, \infty \times Q)  \to \pic(\PP^1 \times Q) \to \pic(\infty \times Q)$$
in which the first arrow is an isomorphism and the last may be identified with a projection $\bb{Z}^3 \to \bb{Z}^2$.
The map $\bb{Z} \cong A_{2,\ol 0}(X) \to A_2(X) \cong \bb{Z} \oplus \bb{Z}$ sends $1$ to $(1,1)$.  
Proposition \ref{isolated sings} yields a pairing (in the notation of (\ref{pervChow}))
$$A_{2,\ol 0}(X) \times A_{1}(X) \to A_0(X) \cong \bb{Z}, \quad (D, \alpha)\mapsto \deg(\cal{O}(D) |_\alpha). $$

The same assignment determines a pairing $A_{2,\ol 0}(X) \times A_{1, \ol 0}(X) \to A_0(X) \cong \bb{Z}$; we proceed to calculate the group $A_{1, \ol 0}(X)$ by a similar procedure.
The birational transform identifies $A_{1,\ol 0}(X) = A_{1,\ol 1}(X)$ with $1$-cycles on $\PP^1 \times Q$ disjoint from $\infty \times Q$, modulo rational equivalences avoiding $\infty \times Q$.  To calculate this group, note that an integral $1$-cycle $C$ disjoint from $\infty \times Q$ must be contained in $p \times Q$ for some $p \neq \infty \in \PP^1$.  Such $1$-cycles $C,C'$ (contained in $p \times Q, p' \times Q$ respectively) are rationally equivalent on $\PP^1 \times Q$ if and only if they are rationally equivalent avoiding $\infty \times Q$.  Since $C \into p \times Q$ can be moved (avoiding $\infty \times Q$) to $0 \times Q$, say, we find $A_{1,\ol 0}(X) \cong A_1(\PP^1 \times \PP^1) \cong \bb{Z} \oplus \bb{Z}$.

The map $\bb{Z} \oplus \bb{Z} \cong A_{1,\ol 0}(X) = A_{1,\ol 1}(X) \to A_{1,\ol 2}(X) = A_1(X) \cong \bb{Z}$ sends both $(1,0)$ and $(0,1)$ to $1$.  The pairing $A_2(X) \times A_{1,\ol 0}(X) \to A_0(X) \cong \bb{Z}$ may be thought of as sending $(D, C)$ to the degree of $\cal{O}(D \cap {X^{sm}}) |_C$ since the Weil divisor $D$ is Cartier along $C \into X^{sm}$.

There are also pairings between divisors.  Intersection with a Cartier divisor determines pairings $ A_{2,\ol 0}(X) \times A_{2, \ol 0}(X) \to A_{1, \ol 0} (X)$ and $A_{2,\ol 0}(X) \times A_{2, \ol 1}(X) \to A_{1, \ol 1}(X)$.  Finally, there is a pairing $A_{2,\ol 1}(X) \times A_{2, \ol 1}(X) \to A_{1, \ol 2}(X)$ which is the closure of the intersection product formed in the smooth locus, given in coordinates by $(a,b),(c,d) \mapsto ad + bc$.
\end{ex}

\begin{ex} More generally, if $Y$ is the cone on a smooth projective variety $X$ of dimension $d-1$, given the stratification $v = Y^{d} = \cdots = Y^1$, we have the following computation of the intersection Chow groups and product of Theorem \ref{thm:star}.  We write $\ol p$ for any perversity with $p_{d} = p$.
There are two types of groups:
\begin{itemize}
\item $A_{r, \ol p}(Y) = A_r(Y) \cong A_{r-1}(X)$ (for $r>0$ and $r - d + p \geq 0$, so that incidence with the vertex is allowed), and
\item $A_{r, \ol  p}(Y) = A_{r, \ol 0}(Y) \cong A_r(X)$ (for $r \geq 0$ and $r - d + p < 0$, so that incidence with the vertex is disallowed).
\end{itemize}

There are three kinds of pairings:
\begin{itemize}
\item $A_{r, \ol p}(Y) \times A_{s, \ol q}(Y) \to A_{r+s-d, \ol p+\ol q}(Y)$, with $p \geq d-r$ and $q \geq d-s$, provided $r+s-d \geq 1$; via the identification above this product is given by the intersection product on $X$:
$$A_{r-1}(X) \times A_{s-1}(X) \xrightarrow{\bu_X} A_{r-1 + s-1 - (d-1)}(X).$$

\item $A_{r, \ol p}(Y) \times A_{s, \ol q}(Y) \to A_{r+s-d, \ol p+ \ol q}(Y)$, with $p < d-r$ and $q \geq d-s$; this is given by

$$A_r(X) \times A_{s-1}(X) \xrightarrow{\bu_X} A_{r+s-1 - (d-1)}(X) \cong A_{r+s-d, \ol 0}(Y)$$
followed by the canonical morphism $A_{r+s-d, \ol 0}(Y) \to A_{r+s-d, \ol p + \ol q}(Y).$

\item  $A_{r, \ol p}(Y) \times A_{s, \ol q}(Y) \to A_{r+s-d, \ol p+ \ol q}(Y)$, with $p < d-r$ and $q < d-s$; this is given by
$$A_r(X) \times A_s(X) \xrightarrow{\bu_X} A_{r+s-(d-1)}(X) \cong A_{r+s-d+1, \ol 0}(Y)$$
followed by intersecting with the Cartier divisor $X \into Y$, which maps $A_{r+s-d+1, \ol 0}(Y)$ to $A_{r+s-d, \ol 0}(Y)$.  {(Note this pairing is not guaranteed by Proposition \ref{isolated sings}.)}

The last pairing is an instance of the following well-known general principle.  If $i: X \into Y$ is a Cartier divisor, and $a,b \in A_*(X)$, then $i_* (a) \cdot i_*(b) = i_* (a \cdot b) \cdot X$ in $A_*(Y)$ provided both sides are defined.  This identity follows from the projection formula, the associativity of the intersection product, and the self-intersection formula \cite[Cor.~6.3]{Ful}.
\end{itemize}

\end{ex}

%%%%%%%%%%%%%%%%%%%%

\section{Generalized cocycles}
\label{sec:perverse}

In this section, $X$  will denote a quasi-projective variety of pure dimension $d$
and $Y$  will denote a quasi-projective variety of pure dimension $n$.  
In Definition \ref{def:cocycle}, we define the cdh-sheaf on
$Sm/k$ of codimension $t$ cocycles of perversity $\ol p$ on $X$ with values in $Y$, $z^{t,\ol p}(X,Y)$.
Following this definition for a general quasi-projective variety $Y$, we shall
often assume that $Y$ is projective so that we can interpret $z^{s,\ol p}(X,Y)$ in terms of
maps to Chow varieties.

We recall that an effective algebraic $t$-cocycle on $X$ with values in $Y$
is the cycle $Z_f \into X \times Y$ associated with some morphism $f: X \to \cC_{n-t}(Y)$.  
Part of the motivation for considering such cocycles is that the $i$-th homotopy group of some 
formulation of the  ``space" of $t$-cocycles on $X$ with values in $\bP^t$
modulo $(t-1)$-cocycles on $X$ with values in $\bP^{t-1}$ represents $H^{2t-i}(X,\bZ(i))$ as in \cite{FV} (or, in
the semi-topological context, $L^tH^{2t-i}(X)$ as in \cite{FLcocycle}).  An important feature of cocycle groups is that
there are natural cup product pairings on cocycle groups and cap product pairings  relating cocycle
groups and cycle groups.

We proceed to develop a theory of ``generalized cocycles" on a stratified variety $X$ with values
in $Y$.  As the name suggests, an effective generalized cocycle is given by weakening the condition 
that it is the graph of some morphism; instead, in the case $Y$ is projective, we require that it be the graph of some rational
map $f: X \dashrightarrow \cC_{n-t}(Y)$.  

One should view generalized cocycles on $X$ as 
cycles (on $X\times Y$ for some $Y$) which are generically equidimensional over $X$
 (i.e., generically satisfy the cocycle condition) and whose failure to be equidimensional over strata of
 $X$ is governed by a perversity $\ol p$.  Thus, there is an additional constraint on a generalized 
 cocycle of a given perversity $\ol p$ to be a generalized cocycle of some perversity $\ol q < \ol p$, with 
 usual cocycles satisfying the full equidimensionality condition.  The cap product pairing of Section \ref{sec:cup-cap}
will show that
 a generalized cocycle of perversity $\ol p$ taken together with a cycle of perversity $\ol q$
 will essentially satisfy the condition $(*,\ol c)$  with $\ol c = \ol p + \ol q$.  As the perversity  condition
 $\ol p$ of  the generalized cocycle is weakened (i.e., as $\ol p$ increases), such a weakened generalized 
 cocycle pairs with the perversity $\ol q$ cycles satisfying a stronger perversity condition
 (i.e., $\ol q$ decreases).
 
 One formal difference between cycle theories and cocycle theories is that one should not expect
 localization in the contravariant variable $X$.  Thus, the proof of the suspension theorem for generalized
 cocycle spaces does not proceed by first considering $X$ projective and then using localization.  Instead,
 one assumes that the covariant variable $Y$ is projective and observes that the constructions of 
 algebraic homotopies as in Section \ref{sec:three} can be employed on Chow varieties of $Y$.

If $X$ is stratified, then $X \times Y$ inherits a stratification from that of $X$, with 
$(X\times Y)^i \equiv X^i \times Y$.   We define the group of perversity $\ol p$ cocycles on $X$
with values in $Y$, 
$$Z^{t, \ol{p}}(X, Y) \ \subseteq \ Z_{d+n-t, \ol{p}}(X \times Y),$$ 
to be the group of $(d +n -t)$-dimensional cycles $\alpha$ on $X\times Y$ with the 
property that for $x \in X^i - X^{i+1}$, the dimension of $| \alpha | \cap (x \times Y)$ is no larger than $n-t+p_i$
(for $i =1, \ldots, d$), { and for $x \in X - X^1$ the dimension of $|\alpha| \cap (x \times Y)$ is $n-t$.}
Because this condition is a constraint on the support $|\alpha|$ of $\alpha$, this does not permit
``large" fibers to cancel.
Roughly speaking, a cycle lies in $Z_{d+n-t, \ol{p}}(X \times Y)$ if its excess with each stratum 
$X^i \times Y$ is not too large; it lies in the smaller group $Z^{t, \ol{p}}(X, Y)$ if in addition this excess 
is distributed evenly over each stratum $X^i - X^{i+1}$.

\begin{defn}
\label{def:cocycle}
Let $\ol p$ be a perversity, and let $t$ be an integer $0 \leq t \leq n$.  We define 
\vspace{-.3cm}
$$z^{t,\ol p}(X,Y) \  \subseteq \ z(X \times Y, d+n-t)_{\ol p} \ \subseteq \  z(X\times Y,d+n-t)$$
to be the subpresheaf (on $Sch / k$) whose value on $U$ consists of $U$-relative cycles
with $\bZ$-coefficients
$W \into U\times X \times Y$ such that for all $u \in U$, the specialization
$W_u \in Z_{d+n-t}(X_u \times Y)$ belongs to $Z^{t, \ol p}(X_u,Y)$.
By allowing $\bZ[1/p]$-coefficients, we obtain the presheaf $z^{t,\ol p}(X,Y) [1/p]$.
We define the subpresheaves $z^{t, \ol p}_{equi} (X,Y) \ \subseteq \  z_{equi} (X \times Y, d+n-t)_{\ol p}$ similarly.

We define {the} bivariant perversity $\ol p$ motivic cohomology group of bidegree $(i,t)$ to be the group
$$H^{i, t, \ol p}(X, Y) \equiv \pi_{2t-i}(z^{t, \ol p}(X, Y)(\bu)).$$  
These groups are contravariantly functorial with respect to flat, stratified morphisms $f: X' \to X$, and covariantly functorial with respect to proper morphisms $g: Y \to Y'$: we have $f^* : H^{i, t, \ol p}(X, Y) \to H^{i, t, \ol p}(X', Y)$ and $g_* : H^{i, t, \ol p}(X,Y) \to H^{i+2r, r +t, \ol p}(X, Y')$, where $r = \dim(Y') - \dim(Y)$.
\end{defn}

\begin{lemma}
\label{hyperplane}
Let $X$ be a stratified quasi-projective variety and let $\ol  p$ be a perversity.  The homotopy class of the map 
$$i_\ell: z^{t-1,\ol p}(X,\bP^{t-1}) (\bu)  \ \to \ z^{t,\ol p}(X,\bP^t) (\bu) $$
induced by the embedding $\ell: \bP^{t-1} \hookrightarrow \bP^t$ of a hyperplane is independent of the choice
of hyperplane $\ell$ (i.e., independent of the choice of linear embedding).

Similarly, the homotopy class of the quotient map 
$$ p_\ell: z^{t,\ol p}(X,\bP^t) (\bu)  \ \to  z^{t,\ol p}(X,\bP^t) (\bu) /z^{t-1,\ol p}(X,\bP^{t-1}) (\bu) $$
is independent of the choice of hyperplane $\ell$.
\end{lemma}

\begin{proof}
Let $\ell, \ell^\prime: \bP^{t-1} \to \bP^t$ be two linear embeddings and let $\theta \in PGL_{t+1}$ satisfy the condition
that $\theta\circ \ell = \ell^\prime$.  Choose a map $f: \bA^1 \to PGL_{t+1}$ with $f(0) = \id, \ f(1) = \theta$. 
The action of $PGL_{t+1}$ on $\bP^t$ and the morphism $f$ determine a morphism $\bb{A}^1 \times \PP^t \to \PP^t$.  Pulling back along this morphism determines a morphism of sheaves
$$\Theta: z^{t,\ol p}(X,\bP^t) (-)  \ \to \ z^{t,\ol p}(X,\bP^t) (- \times \bb{A}^1)$$
such that the composition  
$$\Theta \circ i_\ell : z^{t-1,\ol p}(X,\bP^{t-1}) (-)  \ \to \ z^{t,\ol p}(X,\bP^t) (- \times \bb{A}^1)$$  
is  a homotopy relating $i_\ell$ (restriction to $( - \times \{ 0 \} )$) and $i_{\ell^\prime}$  (restriction to $( - \times \{ 1 \} )$).

To prove the second observation, observe that we have a commutative square
\begin{equation}
\label{theta-commute}
\xymatrix{z^{t,\ol p}(X,\bP^t)(\bu)  \ar[r]^-{p_\ell} \ar[d]_\theta & z^{t,\ol p}(X,\bP^t)(\bu)/z^{t-1,\ol p}(X,\bP^{t-1})(\bu)  \ar[d]_{\ol \theta} \\
z^{t,\ol p}(X,\bP^t)(\bu)  \ar[r]_-{p_{\ell^\prime}} &  z^{t,\ol p}(X,\bP^t)(\bu)/z^{t-1,\ol p}(X,\bP^{t-1})(\bu)  . } 
\end{equation}
Here, $\ol \theta$ is the map on quotients induced by $\theta$; both $\theta, \ \ol \theta$ are isomorphisms.  Since $\theta$
is homotopic to the identity, we conclude that $p_\ell, \ p_{\ell^\prime}$ are homotopic.
\end{proof}

\begin{defn} \label{defn:pervcoh}
First we define a simplicial abelian group
$$z^{t,\ol p}(X) (\bu) := z^{t, \ol p}(X, \PP^t) (\bu) /  z^{t-1, \ol p}(X, \PP^{t-1}) (\bu) ;$$
this is canonical by Lemma $\ref{hyperplane}$.  The perversity $\ol p$ motivic cohomology groups are then defined to be its homotopy groups:
 $$H^{i, t, \ol p}(X) \ \equiv \ \pi_{2t-i}  (z^{t,\ol p}(X) (\bu)).$$

If Voevodsky acyclicity is available, there is a canonical homotopy equivalence
$$ ( z^{t, \ol p}(X, \PP^t) /  z^{t-1, \ol p}(X, \PP^{t-1}) )_{cdh} (\bu)  \stackrel{\sim}{\to} z^{t, \ol p}(X, \PP^t) (\bu) /  z^{t-1, \ol p}(X, \PP^{t-1}) (\bu) $$
and so
one is not really forced to choose between the quotient simplicial abelian group and evaluating the quotient sheaf on $\Deltadot$.
\end{defn}

\begin{remark} \label{FV compare}
If $X$ is smooth, $k$ admits resolutions of singularities, and $\ol p$ is the zero perversity, then we recover the motivic cohomology groups of Friedlander-Voevodsky: $H^{i, t , \ol 0}(X) = H^i(X, \bb{Z}(t)).$  This follows from \cite[Prop.~6.4, Thm.~8.1, Thm.~8.2]{FV}.  
One reason this comparison is likely to fail for singular $X$ is that the zero perversity condition on a cycle does not imply it has well-defined specializations (let alone with universally integral coefficients), whereas the groups $H^i(X, \bb{Z}(t))$ are defined using cycles which have well-defined specializations for all $x \in X$.
If $X$ is smooth, then the zero perversity condition on a cycle (i.e., the condition used to define  $H^{i, t , \ol 0}(X)$) implies it has well-defined specializations for all $x \in X$ by \cite[Cor.~3.4.5]{SV}.
\end{remark}

{The following proposition relates generalized cocycles to Chow varieties when the covariant variable is projective.}

\begin{prop} \label{cocycle basic}
Let $X$ be a stratified quasi-projective variety of dimension $d$ and $Y,\ T$ be
projective varieties of dimension $n,\ m$ respectively.
Let $W \into U \times X \times Y$ be an element of $z^{t, \ol p}(X,Y)(U)$.
\begin{enumerate}

\item  For every $u \in U$, every component of the specialization $W_u$ 
is the closure of the cycle associated to a rational map $f_u : X_u \dashrightarrow \cC_{n-t}(Y)$ defined on $(X-X^1)_u$.

\item For any fat point $(x_0, x_1, R)$ over $u \in U$ there is a rational map $\wt{f} : X_R \dashrightarrow \cC_{n-t}(Y)$ defined on $(X-X^1)_R$ such that the compositions (set $K := \FRAC R$):
$$X_k \xrightarrow{\id_X \times x_0} X_R  \stackrel{\wt{f}}{\dashrightarrow} \cC_{n-t}(Y)  \ \ \text{and} \ \  X_K \to X_R  \stackrel{\wt{f}}{\dashrightarrow} \cC_{n-t}(Y)$$
coincide with
$$X_k \to X_u \stackrel{f_u}{\dashrightarrow} \cC_{n-t}(Y)  \ \  \text{and} \ \   X_K \to X_{\eta_U}  \stackrel{f_{\eta_U}}{\dashrightarrow} \cC_{n-t}(Y).$$

\item
For any continuous algebraic map $g: \cC_{n-t}(Y) \to \cC_{m-s}(T)$, the
closure of the cycle associated to $g \circ f_{\eta_U}: X_{\eta_U} \dashrightarrow \cC_{m-s}(T)$, denoted $W_g$,
is an element of $z^{s,\ol p}(X,T)(U)$.

\item
For any continuous algebraic map $h:  \cC_{n-t}(Y) \times  \bb{A}^1\to \cC_{m-s}(T) $, the
closure of the cycle associated to 
$h \circ (f_{\eta_U} \times \id_{\bb{A}^1}) : X_{\eta_U}  \times \bb{A}^1  \dashrightarrow \cC_{m-s}(T)$, denoted ${(W_{\bb{A}^1})}_h$,
is an element of $z^{t,\ol p}(X,T)(U \times \bb{A}^1)$.  The formation of ${(W_{\bb{A}^1})}_h$ is compatible with restriction to $t \in \bb{A}^1$ in the sense that the image of ${(W_{\bb{A}^1})}_h$ in $z^{t,\ol p}(X,T)(U \times \{ t \})$ coincides with $W_{h_t}$.

\end{enumerate}
\end{prop}

\begin{proof} 
Let $W'_u \into X_u \times Y$ denote a component of the specialization of $W$ at some $u \in U$.
Since $X - X^1$ is smooth, the restriction $W'_u |_{(X - X^1)_{u}}$ is an element of $z_{equi}(Y_{},n-t)((X - X^1)_{u})$ \cite[Cor.~3.4.5]{SV}, and there is a canonical inclusion $z_{equi}(Y_{},n-t)((X - X^1)_{u}) \subseteq \Hom((X - X^1)_{u}, \cC_{n-t}( Y))^+$.  This establishes (1).

The perversity condition implies that for any $u \in U$, all of the generic points of $W_u$ lie in $(X - X^1)_{u} \times Y$, so to verify (2) we may restrict to $X-X^1$, where all of the rational maps are defined.  Since $Y$ is projective the pullbacks on $z_{equi}(Y_{},n-t)$ correspond to composition of morphisms to Chow varieties.

Now we show $W_g$ has well-defined specializations.  
The specializations are determined by the generic points of the cycle ${(W_g)}_{\eta} \into X_\eta \times Y$, where $\eta$ denotes the union $\cup \eta_U$ of the generic points of $U$.  But both $W_{\eta}$ and ${(W_g)}_{\eta}$ have their generic points in $(X - X^1)_\eta \times Y$, so we may restrict to $X - X^1$.  Since specialization corresponds to restriction of morphisms to Chow varieties, the specializations of ${(W_g)}_{\eta}$ are determined by those of $W_\eta$.  Since the latter do not depend on the fat point, the former are independent as well.

To verify (3), it remains to show the perversity condition is preserved.  We may assume $U$ is the spectrum of a field.
Let $X' \into X \times \cC_{n-t}(Y)$ be the graph of the rational map, and let $\pi: X' \to X \ , \ c: X' \to \cC_{n-t}(Y)$ denote the induced morphisms. 
For any $x \in X$ we have the following formulas for the dimensions of the fibers $W_x, \ (W_g)_x$:
$$\dim(W_x) = (n-t) + \dim(\im(c : \pi^{-1}(x) \to \cC_{n-t}(Y))) $$
$$\dim((W_g)_x) = (m-s) + \dim(\im(g \circ c : \pi^{-1}(x) \to \cC_{m-s}(T))) $$
Clearly $\dim(\im(g \circ c : \pi^{-1}(x) \to \cC_{m-s}(T)))  \leq \dim(\im(c : \pi^{-1}(x) \to \cC_{n-t}(Y))) $, so the perversity of $W_g$ is no worse than that of $W$.  The verification of (4) is similar and we omit the details.
\end{proof}

We denote by $z^{t, \ol p}(X, \Sigma(Y))_Y \subset z^{t, \ol p}(X, \Sigma(Y))$ the subpresheaf consisting of $U$-relative cycles $W$ 
none of whose specializations $W_u \into X_u \times \Sigma(Y)$ have components
contained in the Cartier divisor $X_u \times Y \into X_u \times \Sigma(Y)$, and
satisfy the property that $W_u \cap (X_u \times Y)$ belongs to $Z^{t, \ol p}(X_u,Y)$.

In the proof of the following theorem, we employ the same moving constructions which we 
used  in the proof of Theorem \ref{thm:proj-susp}.

\begin{thm} \label{thm:suspgen} Let $X$ be a stratified quasi-projective variety, let $Y$ be a projective variety, 
and let $\ol p$ be a perversity.  
Equip $\Sigma(X)$ with the stratification $\{ \Sigma(X^i) \}$, where $\{ X^i \}$ is the given stratification of $X$.
Fiberwise  suspension  induces homotopy equivalences
$$\Sigma_Y: z_{equi}^{t, \ol p}(X, Y)(\bullet) \  \stackrel{\sim}{\to} \  z_{equi}^{t, \ol p}(X, \Sigma(Y))(\bullet),  $$
$$\Sigma_Y: z_{}^{t, \ol p}(X, Y)(\bullet)  \  \stackrel{\sim}{\to} \   z_{}^{t, \ol p}(X, \Sigma(Y))(\bullet).$$
Therefore we have an induced isomorphism $H^{i, t, \ol p}(X,Y) \cong H^{i, t, \ol p}(X, \Sigma(Y))$.
\end{thm}

\begin{proof}
The overall strategy is similar to that employed in the proof of Theorem $\ref{thm:proj-susp}$: deformation to the normal cone and the projecting cones construction provide $\bb{A}^1$-homotopies and allow us to conclude that each of the morphisms:

\begin{equation} \label{gen cocycle factor} z^{t, \ol p}(X,Y) (\bu) \xrightarrow{\Sigma_Y}  z^{t, \ol p}(X,\Sigma(Y))_Y (\bu) \to z^{t, \ol p}(X, \Sigma(Y)) (\bu) \end{equation}
is a homotopy equivalence.  We explain why the constructions given in the proofs of Propositions $\ref{prop:iso1}$ and $\ref{prop:iso2}$ suffice, and we do not repeat the arguments which require only modification of notation.  We write the proof for  $z_{}^{t, \ol p}(X, Y)$, but the same argument works for $z_{equi}^{t, \ol p}(X, Y)$.

The deformation to the normal cone of Proposition $\ref{prop:iso1}$ defines a continuous algebraic map $\varphi: \cC_{n-t}(\Sigma(Y))_Y \times \bb{A}^1 \to \cC_{n-t}(\Sigma(Y))_Y$.  By Proposition $\ref{cocycle basic}$(4), this provides a morphism:
$$\varphi: z^{t, \ol p}(X, \Sigma(Y))_Y (-) \to z^{t, \ol p}(X, \Sigma(Y))_Y ( - \times \bb{A}^1).$$
Let $\varphi_t$ denote the composition of $\varphi$ with restriction to $( - \times \{ t \})$.  We must show:
\begin{itemize}
\item $\varphi_0$ is the identity,
\item $\varphi_1$ has image contained in $\Sigma_Y(z^{t, \ol p}(X,Y))$, and
\item $\varphi_t$ acts as the identity on $\Sigma_Y(z^{t, \ol p}(X,Y))$ for all $t \in \bb{A}^1$.
\end{itemize}

The morphism $\varphi_0$ is induced by the identity on the Chow variety, and $W = W_{\id}$, so the first property is clear.  The third property follows for a similar reason.

To see that the second property holds, note that any specialization ${(W_{\varphi_1})}_u$ is associated to the rational map $X_u \dashrightarrow \cC_{n-t+1}(\Sigma(Y)) \xrightarrow{\varphi_1} \Sigma_Y (\cC_{n-t}(Y)) \into \cC_{n-t+1}(\Sigma(Y))$.  Therefore ${(W_{\varphi_1})}_u |_{X - X^1}$ is a suspension, and the closure of a suspension is a suspension (namely, it is the suspension of the closure!).  Alternatively, the fiber of ${(W_{\varphi_1})_u}$ over $x \in X$ is the image of ${(W_{\varphi_1 \circ c})_u} \cap (\pi^{-1}(x) \times Y) \to x \times Y$, and $\varphi_1 \circ c : X' \to \cC_{n-t+1} (\Sigma (Y))$ factors through $\Sigma_Y (\cC_{n-t}(Y))$, so all of the fiber cycles of $W_{\varphi_1 \circ c} \to X'$ are suspensions.  The image is therefore a suspension as well.  This proves the generalized cocycles analogue of Proposition \ref{prop:iso1} and establishes that the first arrow in $\ref{gen cocycle factor}$ is a homotopy equivalence.

We proceed to analyze the second arrow in $\ref{gen cocycle factor}$.  The projecting cones are slightly more delicate for the simple reason that $\cC_{n-t+1}(\Sigma(Y))_Y \subset \cC_{n-t+1}(\Sigma(Y))$ is open rather than closed, so that we cannot conclude that $X$ lands in $\cC_{n-t+1}(\Sigma(Y))_Y$ simply because $X-X^1$ does.  The construction of Proposition $\ref{prop:iso2}$ provides a morphism:
$$\psi := \psi_e : z^{t, \ol p}(X, \Sigma(Y), \leq d) (-) \to z^{t, \ol p}(X, \Sigma(Y), \leq de) (- \times \bb{A}^1)$$
where $d$ bounds the degree of the cycles on $Y$ and $e$ depends on $d$.
We must show:
\begin{itemize}
\item $\psi_0$ is $e$ times the identity,
\item $\psi_t$ carries $z^{t, \ol p}(X, \Sigma(Y), \leq d) $ into $z^{t, \ol p}(X, \Sigma(Y), \leq de)_Y$ for general $t \in \bb{A}^1$.
\end{itemize}

We have a morphism $\psi : \cC_{n-t+1, \leq d}(\Sigma(Y)) \times \bb{A}^1 \to \cC_{n-t+1, \leq de}(\Sigma(Y))$ which restricts to a closed immersion (namely, $e$ times the identity) at $t=0$.  Therefore there is an open subscheme $S \subset \bb{A}^1$ such that $\psi_t$ is a closed immersion for $t \in S$ by \cite[Lemma I.1.10.1]{Kol}.  {We may assume $1 \in S$}, and then given $W \in z^{t, \ol p}(X,\Sigma(Y))(U)$, our task is to show $W_{\psi_1} \in z^{t, \ol p}(X, \Sigma(Y))_Y(U)$.

Since $\psi_1$ is a closed immersion, the graph of $X \dashrightarrow \cC_{n-t+1, \leq d}(\Sigma(Y)) \xrightarrow{\psi_1} \cC_{n-t+1, \leq de}(\Sigma(Y))_Y \subset \cC_{n-t+1, \leq de}(\Sigma(Y))$ is isomorphic to the graph $X' \into X \times \cC_{n-t+1, \leq d}(\Sigma(Y))$.  This implies all of the specializations of the cycle $W_{\psi_1} \into U \times X \times \Sigma(Y)$ are covered by (birational, proper) surjections ${(W_{\psi_1 \circ c})}_u \to {(W_{\psi_1})}_u$.  The support of ${(W_{\psi_1 \circ c})}_u$ over some $x' \in X'$ is the cycle $\psi_1 (c(x'))$, and none of these $(n-t+1)$-dimensional cycles are contained in $Y \into \Sigma(Y)$.  Therefore, the cycle ${(W_{\psi_1 \circ c})}_u\cap ( \pi^{-1}(x)_u \times \Sigma(Y))$ is not contained in $X_u \times Y \into X_u \times \Sigma(Y)$.
\end{proof}

We will need the following particular case of the proper push-forward morphism.  If $X$ is a stratified variety and 
$i: Y \into Y'$ is a closed immersion of pure codimension $c$, then the push-forward along $i$ determines a morphism 
of presheaves $z^{t, \ol{p}}(X,Y) \to z^{t+c, \ol{p}}(X, Y')$.  In particular, the inclusion of a hyperplane $i: \PP^{s-1} \into \PP^s$ 
induces a morphism $i_* : z^{s-1, \ol{p}}(X,\PP^{s-1}) \to z^{s, \ol{p}}(X, \PP^s)$ of presheaves on $Sch / k$.  The existence 
of $i_*$ follows from the existence of proper push-forward functors on the presheaves $z(X,r)$ and $z_{equi}(X,r)$ \cite[Cor.~3.6.3]{SV}.  Alternatively, $i_*$ is the morphism provided by Proposition $\ref{cocycle basic}$(3) for the continuous algebraic map $\cC_0(\PP^{s-1}) \to \cC_0(\PP^s)$.

\begin{lemma}
\label{lem:com}
Let $X$ be a stratified quasi-projective variety, and let $\ol p$ be a perversity. 
The following square is homotopy commutative:
\begin{equation}
\label{sigma-commute}
\xymatrix{z^{t-1,\ol p}(X,\bP^{t-1})(\bu)  \ar[r] \ar[d]_{\Sigma^i} & z^{t,\ol p}(X,\bP^t)(\bu)  \ar[d]_{\Sigma^i} \\
z^{t-1,\ol p}(X,\bP^{t+i-1})(\bu)  \ar[r] &  z^{t,\ol p}(X,\bP^{t+i})(\bu). } 
\end{equation}
\end{lemma}

\begin{proof}
The two compositions of the square (\ref{sigma-commute}) are given  by first embedding $\bP^{t-1}$ in $\bP^t$,
then suspending $i$-times; and by first suspending $i$-times, then embedding $\bP^{t+i-1}$ in $\bP^{t+i}$.
These are readily seen to be related by an $\bA^1$-family of automorphisms of $\bP^{t+i}$, and
the required homotopy is obtained by composing with
these automorphisms.
\end{proof}

\begin{thm}
\label{iterated suspension}
Let $X$ be a stratified quasi-projective variety, and let $\ol p$ be a perversity.  The fiberwise suspension map 
(with respect to $\bP^t$) induces a homotopy equivalence

$$z^{t,\ol p}(X)(\bullet) \ \stackrel{\sim}{\to} \ z^{t,\ol p}(X,\bP^{t+i})(\bullet) / z^{t-1,\ol p}(X,\bP^{t+i-1})(\bullet).$$
\end{thm}

\begin{proof}
This follows from Theorem \ref{thm:suspgen}, by applying the 5-Lemma to  the map of short exact sequences 
(arising from Definition \ref{defn:pervcoh}) of the form 
$$0 \to z^{t-1,\ol p}(X,\bP^{t-1})(\bullet) \to  z^{t,\ol p}(X,\bP^t)(\bullet)  \to z^{t,\ol p}(X)(\bullet) \to 0$$
determined by Lemma \ref{lem:com}.
\end{proof}

Two natural sources of cocycles are flat morphisms and vector bundles.  Here we explain how arbitrary morphisms and coherent sheaves give rise to generalized cocycles (for a stratification and perversity determined by the morphism and sheaf respectively).

{\bf Morphisms.} Let $X$ and $Y$ be quasi-projective $k$-varieties.  If 
$f: Y \to X$ is a dominant flat morphism, then taking the cycle associated to the 
scheme-theoretic fiber $f^{-1}(x)$ determines an effective $d$-cocycle on $X$ with values in $Y$.
As we see in the following example, general morphisms provide examples of generalized cocycles.

\begin{ex}
With the notation as above, we define
$$\bZ \Hom^{\ol p}(-\times Y,X) \ \subset \ z^{d,\ol p}(X,Y)(-)$$
to be the subsheaf whose value on $U$ is the free abelian group on
the morphisms $f: U \times Y \to X$ with the property that the 
induced map $f_u: Y_u \to X_u$ is dominant
and the transpose of the graph
$\Gamma_{f_u}^{t} \subset X_u \times Y_u$ lies in $Z^{d,\ol p}(X_u, Y_u)$  (for  all $u \in U$).
\end{ex}

Our next proposition shows how $\bZ\Hom^{\ol p}(Y,X)$ acts on generalized cocycles.

\begin{prop} \label{morphisms act}
Let $X,\ Y$  be projective varieties, let $W$ be a quasi-projective variety, suppose $X$ is stratified, and
let $\ol p$ be a perversity.
Then there is a natural pairing given by proper push-forward
$$\bZ \Hom^{\ol p}(Y,X) \times z^{t,\ol 0}(Y,W) \to z^{d+t-n,\ol p}(X,W).$$
\end{prop}
\begin{proof}
It suffices to define the pairing for a pair $(f, \beta) \in \Hom^{\ol p}(Y,X)(U) \times z^{t,\ol 0}(Y,W)(U)$ consisting of a morphism $f :U \times Y \to X$ and a cycle $\beta \into U \times Y \times W$ with specializations $\beta_u$ equidimensional over $Y_u$.  Now $f$ induces a proper morphism $f : U \times Y \times W \to U \times X \times W$, and we claim $f_* (\beta)$ belongs to $z^{t,\ol p}(X,W)(U)$.  Set $w = \dim (W)$.  By hypothesis, for any $(u,y) \in U \times Y$, we have $\dim( |\beta_u|_y ) =w-t$.  Therefore, for any $(u, x) \in U \times X$, we have $\dim ( |f_*(\beta)_u|_x ) \leq \dim (f^{-1}(x)) + w-t$.  By assumption, $x \in X^i - X^{i+1}$ implies $\dim (f^{-1}(x)) \leq (n-d) + p_i$, and the claim follows.
The formation of $f_* (\beta)$ is functorial in $U$, so the pairing defines a natural transformation.
\end{proof}

 {As mentioned in the introduction, cycle classes on a resolution determine generalized cocycles on the variety being resolved.  We say a morphism $f : Y \to X$ determines a stratification $S$ and perversity $\ol p$ if $f$ does not belong to $\Hom^{\ol {q}}(Y,X)$ for any stricter incidence datum $(T, \ol q)$, with $T$ a stratification.
 \begin{prop}
 If $f: Y \to X$ is a resolution of singularities, push-forward along $f$ defines a morphism $H_{2 n - i}^{BM}(Y, \bb{Z}(n- t)) \to H^{i, t, \ol p}(X)$ {for the stratification and perversity determined by the resolution (and hence for any less strict incidence datum).}  
  \end{prop}
\begin{proof} We have a push-forward $f_*: H^{i, t, \ol 0}(Y) \to H^{i, t, \ol p}(X)$ by Proposition $\ref{morphisms act}$, an identification $H^{i, t, \ol 0}(Y) \cong H^i (Y, \bb{Z}(t))$ by Remark $\ref{FV compare}$, and Friedlander-Voevodsky duality $H^i(Y, \bb{Z}(t)) \cong H^{BM}_{2 n - i }(Y, \bb{Z}(n -t))$ \cite[Thms.~8.2, 8.3(1)]{FV}. \end{proof}}

\textbf{Coherent sheaves.}  Suppose $\cF$ is a globally generated coherent sheaf on $X$ with generic rank $r$.  There is an exact sequence of sheaves on $X$:
$$0 \to \cK \to H^0(X, \cF) \otimes_k \cO_X \to \cF \to 0.$$
If $U \subset X$ is the locus over which $\cF$ is locally free, then the projectivization of the locally free sheaf $\cK |_U$ may be viewed as an element of $Z^r(U, \PP^n)$ with $n = h^0(X, \cF)-1$. 

We shall show in Proposition $\ref{coh1}$ below that the closure in $X \times \PP^n$ of this $\PP^{r-1}$-bundle over $U$, denoted $\PP(\cK)$, is an element of $Z^{r, \ol{p}}(X, \PP^n)$ for a stratification and perversity which may be expressed in terms of $\cF$ itself.
Namely, stratify $X$ according to the rank-jumping behavior of $\cF$.  Then there exists a sequence of integers $p_1, \ldots, p_d$ such that $x \in X^i$ if and only if $\rk (\cF |_x ) \geq r + p_i$ and $x \in X^i - X^{i+1}$ if and only if $\rk (\cF |_x ) \leq r + p_i$.  We say this stratification and perversity are determined by $\cF$.

\begin{prop} \label{coh1} Let $\cF$ be a globally generated coherent sheaf on $X$ with generic rank $r$, and set $n = h^0(X, \cF)-1$.  
Then $\PP(\cK) \in Z^{r, \ol{p}}(X, \PP^{n})$ for the stratification and perversity determined by $\cF$. \end {prop}

\begin{proof} Let $\PP(\cF) \into X \times \PP^n$ denote the closure in $X \times \PP^n$ of the $\PP^{r-1}$-bundle over $U$ classified by the surjection
$H^0(X, \cF) \otimes_k \cO_U \to \cF |_U$; the $\PP^n$ which appears here is dual to the one which houses $\PP(\cK)$.  Then the fiber of $\PP(\cF)$ over $x \in X$ is contained in the projectivization of the vector space $\cF |_x$, in fact $\PP(\cF)$ is the main component of the (possibly reducible) projectivization of $\cF$, hence the perversity of $\PP(\cF)$ is controlled by the rank-jumping behavior of $\cF$.

To prove the lemma, then, it suffices to show the perversity of $\PP(\cK)$ is identical to that of $\PP(\cF)$.  Let $X' \into X \times \bb{G}(n-r, n)$ denote the graph of the rational map $X \dashrightarrow \bb{G}(n-r, n)$ determined by $\cK |_U$, and let $X'' \into X \times \bb{G}(r-1, n)$ denote the graph of the map determined by $\cF |_U$.

Note that $\PP(\cK)$ is the push-forward via $X' \to X$ of the codimension $r$ cocycle on $X'$ with values in $\PP^n$ classified by the morphism $X' \to \bb{G}(n-r, n)$, and similarly $\PP(\cF)$ is the push-forward via $X'' \to X$ of the cocycle determined by $X'' \to \bb{G}(r-1, n)$.  Furthermore $X' \cong X''$ via the isomorphism $\bb{G}(n-r, n) \cong \bb{G}(r-1, n)$.

Let $F'_x \into X', F''_x \into X''$ denote the fibers over $x \in X$.  The dimension of the fiber of $\PP(\cK)$ over $x \in X$ is equal to the dimension of the image of the morphism $F'_x \to X' \to \bb{G}(n-r, n)$ plus $n-r$.  Similarly the dimension of the fiber of $\PP(\cF)$ over $x \in X$ is equal to the dimension of the image of $F''_x \to X'' \to \bb{G}(r-1, n)$ plus $r-1$.  By the previous paragraph, $F'_x  \cong F''_x$ compatibly with the isomorphisms of Grassmannians, hence the perversities agree.  \end{proof}

We denote by $z^{r, \ol{p}}(X, \PP^\infty)( \bullet)$ the simplicial abelian group $\displaystyle \colim_n z^{r, \ol{p}}(X, \PP^n)( \bullet)$.  Note that the transition maps in the colimit are the suspension weak equivalences
$\Sigma_{\PP^n} : z^{r, \ol{p}}(X, \PP^n)( \bullet) \to z^{r, \ol{p}}(X, \PP^{n+1})( \bullet).$

\begin{prop} The class of $\PP(\cK)$ in $\pi_0 (z^{r, \ol{p}}(X, \PP^\infty)( \bullet))$ is independent of the choice of generating sections of $\cF$.  \end{prop}

\begin{proof} Suppose given exact sequences
$$0 \to \cK_f \to H^0(X, \cF) \otimes_k \cO_X \xrightarrow{f} \cF \to 0,$$
$$0 \to \cK_g \to H^0(X, \cF) \otimes_k \cO_X \xrightarrow{g} \cF \to 0.$$

The section $t \cdot f + (1-t) \cdot g$ determines an exact sequence of coherent sheaves on $X \times \bb{A}^1$ (let $p : X \times \bb{A}^1 \to X$ denote the projection):
$$0 \to \cK_{\bb{A}^1} \to H^0(X, \cF) \otimes_k \cO_{X \times \bb{A}^1}^2 \to p^* \cF \to 0.$$

The perversities of $\PP(p^* \cF)$ and $\PP(\cK_{\bb{A}^1})$ agree by the argument in the previous proposition, and the perversity of $\PP(p^* \cF)$ (for the product stratification) is the same as that of $\PP(\cF)$ itself.  Therefore $\PP(\cK_{\bb{A}^1})$ belongs to $z^{r, \ol{p}}(X, \PP^{2n+1})(\Delta^1)$.   Furthermore $\PP(\cK_{\bb{A}^1})_0 = \Sigma^{n+1} \PP(\cK_g)$ and $\PP(\cK_{\bb{A}^1})_1 = \Sigma^{n+1} \PP(\cK_f)$ since any additional components in the fibers at $t=0,1$ would violate the perversity condition, hence the elements agree in $\pi_0 (z^{r, \ol{p}}(X, \PP^{2n+1})(\bullet))$.  \end{proof}

%%%%%%%%%%%%%%%%%%%%%%%%%%%%%%%%%%
%%%%%%%%%%%%%%%%%%%%%%%%%%%%%%%%%%%
%
%

\vskip .2in

\section{Join and cup product}
\label{sec:cup-cap}

{In this final section, we define pairings on sheaves of generalized cocycles and/or sheaves of perversity cycles.  These pairings determine pairings on the perversity motivic cohomology of Definition \ref{defn:pervcoh} and perversity motivic homology of Definition \ref{defn:pervhom} by Remark \ref{homotopy pairing}.}

The geometric operation underlying our cup product is the join.  The semi-topological precursor (in the absence of perversities)
 of our product is the cup product pairing on semi-topological cohomology defined using the fiberwise join \cite[Thm.~6.1]{FLcocycle};
 building on this, an algebraic version for smooth varieties is developed in \cite[Prop.~8.6]{FV}.

\begin{defn} Let $V$ be a $k$-scheme.  Given $\alpha \into V \times \bb{P}^t$ and $\beta \into V \times \bb{P}^s$, let $J_V(\alpha, \beta) \into V \times \bb{P}^{t+s+1}$ denote their fiberwise join.  If $\alpha$ (resp.~$\beta$) is an integral subscheme whose ideal sheaf is locally generated by $\{ f(x,t) \}$ (resp.~$\{ g(x,s) \}$), then $J_V(\alpha, \beta)$ is the (integral) subscheme with ideal sheaf locally generated by $\{ f(x,t); g(x,s) \}$.  (Here the $x$'s are coordinates on $V$, the $t$'s are coordinates on $\bb{P}^t$, and the $s$'s are coordinates on $\bb{P}^s$.)  We define the join of a general pair of cycles $\alpha, \beta$ by linear extension.  
\end{defn}

\begin{prop} \label{join} Let $X$ be a stratified quasi-projective variety.  The join defines a morphism of functors on $Sch / k$
$$z^{s, \ol{p}}(X, \bb{P}^s) \times z^{t, \ol{q}}(X, \bb{P}^t) \to z^{s+t, \ol{p} + \ol{q}}(X, \bb{P}^{s+t+1})$$
and similarly for the $equi$-theory.
\end{prop}

\begin{proof} We send the pair $(\alpha, \beta) \in z^{s, \ol{p}}(X, \bb{P}^s)(U) \times z^{t, \ol{q}}(X, \bb{P}^t)(U)$ to the fiberwise join $J := J_{U \times X}(\alpha, \beta) \into U \times X \times \PP^{s+t+1}$ described above.  

The join defines a continuous algebraic map $\cC_0(\PP^s) \times \cC_0(\PP^t) \to \cC_1 (\PP^{s+t+1})$ determined by sending $(p,q)$ to the line connecting $i_s(p)$ and $i_t(q)$, where $i_s$ (resp.~$i_t$) identifies $\PP^s$ (resp.~$\PP^t$) with the ``first" $s+1$ (resp.~``last" $t+1$) coordinates of $\PP^{s+t+1}$ {\cite[(6.1.1)]{FLcocycle}.}

The generic points of the join are in one-to-one correspondence with pairs of generic points of the cycles being joined.  Since the generic points of $\alpha_{\eta_U}$ and $\beta_{\eta_U}$ lie in $(X-X^1)_{\eta_U}$, the same is true of $J_{\eta_U}$.  Therefore it suffices to show the restriction of $J$ to $X-X^1$ has well-defined specializations for all $u \in U$.  But on $X-X^1$, all of the specializations $\alpha_u \ , \ \beta_u$ are given by morphisms $f_u : (X - X^1)_u \to \cC_0(\PP^s) \ , \ g_u : (X - X^1)_u \to \cC_0(\PP^t)$.  

Therefore, on $X-X^1$, every specialization $J_u$ is the cycle determined by the morphism $f_u \# g_u : (X-X^1)_u \to \cC_0(\PP^s) \times \cC_0(\PP^t) \to \cC_1 (\PP^{s+t+1})$.  The basic compatibility of morphisms to Chow varieties and pullbacks of cycles (as discussed in the proof of Proposition $\ref{cocycle basic}$) implies $J$ has well-defined specializations.  
From the definition it is clear that the join preserves integrality of the cycle coefficients.

Now we verify $J$ has the required incidence properties, which is a pointwise condition on $U$.  
The relative join is compatible with base change \cite[Remark 1.3.3(2)]{FOVjoins}. 
Therefore, if $x \in X$, the support of $J(\alpha_u, \beta_u)_x \into u \times x \times \bb{P}^{t+s+1}$ coincides with the support of $J ( |\alpha_{u}|_x, |\beta_{u}|_x) \into u \times x \times \bb{P}^{t+s+1}$.  In particular if $\alpha_u \in Z^{s, \ol{p}}(X_u, \bb{P}^s)$ and $\beta_u \in Z^{t, \ol{q}}(X_u, \bb{P}^t)$, then the dimension of the fiber of $J_u$ over $x \in X^i - X^{i+1}$ is less than or equal to $p_i + q_i + 1$, as desired.  
\end{proof}

Next we relate the ``total" groups $z^{s,\ol p} (X,\bP^s)(\bullet)$ to the ``pure" 
groups $z^{i,\ol p}(X)(\bullet)$ (Definition \ref{defn:pervcoh}) which isolate the cycles on $X\times \bP^s$ 
with no component supported on a hyperplane.  The proof here follows closely the proof of \cite[Thm.~2.10]{FLcocycle}.

For positive integers $s,t$ with $s \geq t$, and $K$ algebraically closed, there is a morphism 
$$\pi: SP^s(\PP^1_K) \  \to \ SP^{\binom{s}{t}} (SP^t (\PP^1_K))$$
sending the cycle $\sum_{i \in I} z_i$ to the cycle $\sum_{J \subset I, |J|=t} (\sum_{j \in J} z_j )$.  By Galois descent,
 the same formula defines a morphism assuming that $K$ is perfect, or if one works with cycles with 
 $\bZ[1/p]$-coefficients instead of  $\bZ$-coefficients.  (In characteristic zero, one should ignore all instances of $1/p$ which appear in the statements below.)
Since the symmetric product $SP^m(X)$ of a normal variety $X$ is normal, {the symmetric products which appear coincide with the weak normalizations of the Chow varieties $\cC_{0, m}(X)$.} 
Therefore $\pi$ induces a continuous algebraic map $\pi: \cC_0 (\PP^s) \to \cC_0 (\PP^t)$.

\begin{prop}
\label{prop:symm}
Let $X$ be a stratified quasi-projective variety.   For every $t \leq s$, there are natural maps 
of presheaves 
$$z^{s,\ol p}(X,\bP^s)[1/p](-) \ \to \ z^{t,\ol p}(X,\bP^t)[1/p](-)$$
with the property that for any choice of linear embeddings $\bP^{t-1} \subset \bP^t \subset \bP^s$
the composition 
$$z^{t,\ol p}(X,\bP^t)[1/p](\bu) \ \to \ z^{s,\ol p}(X,\bP^s)[1/p](\bu) \ \to \ z^{t,\ol p}(X,\bP^t)[1/p](\bu) \ \to \ z^{t,\ol p}(X)[1/p](\bu)$$
is homotopy equivalent to the natural projection of Lemma \ref{hyperplane}.
\end{prop}

\begin{proof}
Proposition $\ref{cocycle basic}$(3) implies that $\pi$ induces, for $s \geq t$, a natural transformation $p: z^{s, \ol p}(X, \PP^s)[1/p] \to z^{t, \ol p}(X, \PP^t)[1/p]$.
The flag $\PP^0 \into \PP^1 \into \cdots \into \PP^s$ induces a nested sequence of presheaves:
$$z_{}^{0, \ol{p}}(X, \PP^0)[1/p] \subset z_{}^{1, \ol{p}}(X, \PP^1)[1/p] \subset \ldots \subset z_{}^{s, \ol{p}}(X, \PP^s)[1/p].$$
It suffices to show the composition $p \circ i : z_{}^{t, \ol{p}}(X, \PP^t)[1/p] \subset z_{}^{s, \ol{p}}(X, \PP^s)[1/p] \to z_{}^{t, \ol{p}}(X, \PP^t)[1/p]$ is equal to $\id + \psi$, where $\psi: z_{}^{t, \ol{p}}(X, \PP^t)[1/p] \to z_{}^{t, \ol{p}}(X, \PP^t)[1/p]$ is a morphism factoring through $z_{}^{t-1, \ol{p}}(X, \PP^{t-1})[1/p]$.
For any $\alpha \in z^{t, \ol{p}}(X, \PP^t)[1/p](U)$ and any $u \in U$, the specialization $\alpha_u$
restricts to a cocycle on $(X-X^1)_u$ with values in $\PP^t$.
It follows from \cite[Lemma 2.11]{FLcocycle} that the restriction $j^* ((p \circ i) (\alpha_u) - \alpha_u)$ of $(p \circ i) (\alpha_u) - \alpha_u$ to $(X-X^1)_u$ lies in $(X-X^1)_u \times \PP^{t-1}$.  The morphism $p \circ i$ is compatible with the open immersion $j: X- X^1 \subset X$, and $X-X^1$ contains all of the generic points of $(p \circ i) (\alpha_u) - \alpha_u$.
Therefore the closure of $j^* ((p \circ i) (\alpha_u) - \alpha_u)$, namely $(p \circ i)(\alpha_u) - \alpha_u$, is contained in $X_u \times \PP^{t-1}$. 
\end{proof}

\begin{thm}
\label{thm:splitting}
Let $X$ be a stratified quasi-projective variety. 
The maps of Proposition \ref{prop:symm} induce a homotopy equivalence
\begin{equation}
\label{eq:split}
z_{}^{s,\ol p} (X,\bP^s)[1/p](\bullet) \ \stackrel{\sim}{\longrightarrow} \ \prod_{i=0}^{s} z_{}^{i,\ol p}(X)[1/p](\bullet)
\end{equation}
which is functorial with respect to flat, stratified morphisms.
\end{thm}

\begin{proof} The evaluation of the nested sequence of presheaves at $\Deltadot$ induces a nested sequence of simplicial abelian groups:
$$z_{}^{0, \ol{p}}(X, \PP^0)[1/p](\bullet) \subset z_{}^{1, \ol{p}}(X, \PP^1)[1/p](\bullet) \subset \cdots \subset z_{}^{s, \ol{p}}(X, \PP^s)[1/p](\bullet).$$
Proposition $\ref{prop:symm}$ implies that the formal hypotheses of \cite[Prop.~2.13]{FLcocycle} are satisfied.
The construction involves only the ``targets" $\PP^0, \ldots, \PP^s$, hence are compatible with flat pull-back via stratified morphisms.  
\end{proof}

\begin{remark} \label{itersusp rmk} One can replace the $\PP^s$ on the left hand side of the weak equivalence of Theorem $\ref{thm:splitting}$ with $\PP^r$ (for any $r \geq s$) by appealing to the suspension theorem $\ref{iterated suspension}$. \end{remark}

\begin{prop}
\label{prop:short}
Choose a hyperplane $\bP^{s-1} \hookrightarrow \bP^s$ and a non-negative integer $m$.  Then there is a 
split short exact sequence of homotopy groups
$$0 \to \ \pi_m(z^{s-1,\ol p}(X,\bP^{s-1})[1/p](\bu) ) \to \pi_m(z^{s,\ol p}(X,\bP^s)[1/p](\bu)) \to \pi_m (z^s(X)[1/p](\bu)) \to 0.$$
\end{prop}

\begin{proof}
The short exact sequence of simplicial abelian groups
$$ 0 \to \ z^{s-1,\ol p}(X,\bP^{s-1}) [1/p](\bu)  \to z^{s,\ol p}(X,\bP^s)[1/p](\bu) \to  {z^{s,\ol p}(X,\bP^s) [1/p](\bu) \over z^{s-1,\ol p}(X,\bP^{s-1})[1/p](\bu)} \to 0$$
induces a long exact sequence in homotopy groups (because a surjective homomorphism
of simplicial abelian groups is a Kan fibration).   This long exact sequence splits into split short exact sequences
as asserted thanks to Theorem \ref{thm:splitting}.  \end{proof}

\begin{thm}
\label{cupproduct}
The fiberwise join pairings of Proposition \ref{join}
determine  natural (with respect to $X$)   ``cup product pairings"
$$\cup : H^{i,s,\ol p}(X)[1/p] \otimes H^{j,t,\ol q}(X)[1/p] \ \to \ H^{i+j, s+t,\ol p + \ol q}(X)[1/p].$$
\end{thm}

\begin{proof}
Consider the composition 
$$ \pi_{2s-i}(z^{s, \ol{p}}(X, \bb{P}^s) [1/p](\bu) ) \otimes \pi_{2t-j}(z^{t, \ol{q}}(X, \bb{P}^t)[1/p](\bu) ) \to \hspace{1.5in} $$
$$ \  \to \pi_{2(s+t) -i-j}(z^{s+t, \ol{p} + \ol{q}}(X, \bb{P}^{s+t+1})[1/p](\bu) ) \to  \hspace{1.6in}$$
$$ \ \to \ \pi_{2(s+t) -i-j}(z^{s+t, \ol{p} + \ol{q}}(X, \bb{P}^{s+t+1}) [1/p](\bu) /z^{s+t-1, \ol{p} + \ol{q}}(X, \bb{P}^{s+t})[1/p](\bu) )$$
given by the map induced by fiberwise join followed by the projection.
We consider $\PP^s \# \PP^{t-1}$ and $\PP^{s-1} \# \PP^t$ inside $\PP^s \# \PP^t = \PP^{s+t+1}$ and 
apply the short exact sequence of Proposition \ref{prop:short}
and the independence statement of Lemma \ref{hyperplane}.
It follows that the composition sends both 
$$\pi_{2s-i}(z^{s, \ol{p}}(X, \bb{P}^s)[1/p](\bu)) \otimes \pi_{2t-j}(z^{t-1, \ol{q}}(X, \bb{P}^{t-1})[1/p](\bu)) \hspace{.6cm} \text{and}$$
$$\pi_{2s-i}(z^{s-1, \ol{p}}(X, \bb{P}^{s-1})[1/p](\bu)) \otimes \pi_{2t-j}(z^{t, \ol{q}}(X, \bb{P}^t)[1/p](\bu))$$
to 0.  Now the pairing is obtained by applying the equivalence of the suspension theorem \ref{iterated suspension} (as in Remark \ref{itersusp rmk}).
\end{proof}

The following proposition can be seen as having its origins in a semi-topological version given in \cite[Thm 2.6]{FW}.  Recall that $\ol t$ denotes the top perversity.

\begin{proposition} \label{smooth pairing} Let $X$ be a stratified quasi-projective variety, and let $Y$ be a smooth quasi-projective variety of dimension $n$.  Let $\ol p$ and $\ol q$ be perversities such that $\ol p + \ol q \leq \ol t$.  Restriction of correspondences determines a morphism of presheaves:
$$z^{t, \ol p}(X,Y) \times z(X, r)_{\ol q} \to z(X \times Y, r+n-t)_{\ol p + \ol q}$$
and therefore a pairing:
$$H^{i,t, \ol p}(X, Y) \otimes H^{\ol q}_m (X, \bb{Z}(r)) \to H^{BM}_{2n+m-i}(X \times Y, \bb{Z}(r+n-t)).$$
\end{proposition}

\begin{proof} Given $\alpha, \beta \in z^{t, \ol p}(X,Y)(U) \times z(X, r)_{\ol q}(U)$, the dimension of $\alpha_u$ over any $x \in X^i - X^{i+1}$ is less than or equal to $n-t+p_i$.  The dimension of $\beta_u \cap X^i_{u}$ is less than or equal to $r -i + q_i$.  Therefore the support of $|\alpha| \cap |\beta \times Y| \cap (X^i - X^{i+1} \times Y)$ has dimension no larger than $(r-i+q_i) + (n-t +p_i) = (r+n-t) -i + (p_i + q_i)$.   {This means precisely the pair $(\alpha, \beta)$ satisfies the condition $(*, \ol p + \ol q)$ of Definition \ref{def:star}.}  Then Theorem $\ref{thm:star}$ implies the closure of the intersection product formed in $X^{sm} \times Y$ belongs to $z(X \times Y, r+n-t)_{\ol p + \ol q}(U)$, as desired. \end{proof}

\begin{prop} \label{cap} {Assume $k$ admits resolution of singularities, or that $k$ is perfect and we use $1/p$ coefficients.}
Let $X$ be a stratified quasi-projective variety, and let $\ol p$ and $\ol q$ be perversities such that $\ol p + \ol q \leq \ol t$.
There is a cap product map (morphism of presheaves)
$$z^{t, \ol p}(X) \times z (X,r)_{\ol q} \to z(X \times \bb{A}^t)$$
which induces a {homotopy} pairing of simplicial abelian groups
$$z^{t, \ol p}(X)(\bu) \times z(X,r)_{\ol q}(\bu)  \to z(X, r-t)(\bu)$$
and therefore a pairing
{$$\cap : H^{ i, t, \ol p}(X) \otimes H^{\ol q}_m(X, \bb{Z}(r)) \to H^{BM}_{m-i} (X , \bb{Z}(r-t)).$$}
 \end{prop}
 
\begin{proof} Recall that $z^{t, \ol p}(X)$ is the quotient presheaf 
$z^{t, \ol p}(X, \PP^t) / z^{t-1, \ol p}(X, \PP^{t-1})$.  The pairing of Proposition $\ref{smooth pairing}$ induces a pairing
$$z^{t, \ol p}(X) \times z (X,r)_{\ol q} \to z(X \times \PP^t, r) / z(X \times \PP^{t-1}, r)$$
and we have homotopy equivalences  
$z(X \times \PP^t, r)(\bu) / z(X \times \PP^{t-1}, r)(\bu) \stackrel{\sim}{\to}  z(X \times \bb{A}^t, r)(\bu) \stackrel{\sim}{\leftarrow}  z(X, r-t)(\bu)$ {\cite[Thm.~8.3(1)]{FV}}.
\end{proof}

\begin{remark} If one could establish homotopy equivalences
$$z(X \times \PP^t, r)_{\ol p + \ol q}(\bu) / z(X \times \PP^{t-1}, r)_{\ol p + \ol q}(\bu) \stackrel{\sim}{\to}  z(X \times \bb{A}^t, r)_{\ol p + \ol q}(\bu) \stackrel{\sim}{\leftarrow}  z(X, r-t)_{\ol p + \ol q}(\bu)$$
then Proposition $\ref{cap}$ could be refined in that the targets could be replaced by the perverse versions (with perversity $\ol p + \ol q$).  
Proposition $\ref{cap}$ is the ``correct" statement for complementary perversities (i.e., $\ol p + \ol q = \ol t$).
\end{remark}

Proposition $\ref{smooth pairing}$ extends to the case where $Y$ is singular.

\begin{corollary} \label{sing pairing} Let $X$ be a stratified quasi-projective variety, and let $Y$ be a quasi-projective variety of dimension $n$.
Let $\ol p$ and $\ol q$ be perversities such that $\ol p + \ol q \leq \ol t$.
Restriction of correspondences determines a morphism of presheaves:
$$z^{t, \ol p}(X,Y) \times z(X, r)_{\ol q} \to z(X \times Y, r+n-t)_{\ol p + \ol q}.$$  \end{corollary}
\begin{proof} Embed $Y$ as a closed subvariety of codimension $c$ of some open subvariety $\PP$ of a projective space.  The restriction of the pairing 
$$z^{t+c, \ol p}(X,\PP) \times z(X, r)_{\ol q} \to z(X \times \PP, r+\dim(\PP)-t-c)_{\ol p + \ol q}$$
provided by Proposition $\ref{smooth pairing}$ to the subpresheaf $z^{t, \ol p}(X,Y) \times z(X, r)_{\ol q}$ factors through $z(X \times Y, r+n-t)_{\ol p + \ol q}$. \end{proof}

\begin{remark}
The restriction of the pairing of Corollary $\ref{sing pairing}$ to the subsheaf $\bZ \Hom^{\ol p}(Y,X) \subset z^{d, \ol  p}(X, Y)$ 
may be thought of as sending a pair $(f, \beta) \in \Hom^{\ol p}(Y,X)(U) \times z(X,r)_{\ol q}(U)$ to the pull-back of $\beta \into U \times X$ along $f : U \times Y \to U \times X$.  (Strictly speaking we intersect the graph of $f$ with the pull-back of $\beta$ to $U \times X \times Y$.)
\end{remark}

We establish the compatibility of our pairings with those defined by Goresky-MacPherson.  First we construct the analogue of the perverse cycle class map of Proposition $\ref{prop:GM}$.  In the next two statements, $\bb{F}$ denotes an arbitrary coefficient field.

\begin{lemma}
\label{lem:cohom cycle class}
Let $X$ be a stratified variety of dimension $d$ over $\bb{C}$, and suppose the stratification is sufficiently fine to compute the intersection homology groups $IH^{\ol p}_*(X)$.  Then there is a canonical perverse cycle class map
$$c : H^{2t,t,\ol p}(X) \to IH^{\ol p}_{2(d-t)}(X, \bb{F}).$$
\end{lemma}

\begin{proof}
By applying $\pi_0 ( - (\bu))$ to the inclusion of sheaves $z^{t, \ol p}(X, \PP^t) \to z(X \times \PP^t, d)_{\ol p}$
and composing with the map from Proposition $\ref{prop:GM}$, we obtain a map
$Z^{t, \ol p}(X, \PP^t) / \sim_{\ol p} \to IH^{\ol p}_{2d}(X \times \PP^t, \bb{F})$.
This construction is functorial with respect to the inclusion of a hyperplane $\PP^{t-1} \to \PP^t$, hence it yields
$$H^{2t,t,\ol p}(X) \to IH^{\ol p}_{2d}(X \times \PP^t, \bb{F}) / IH^{\ol p}_{2d}(X \times \PP^{t-1}, \bb{F}).$$

Let $[\PP^j] \in H_{2j}(\PP^t, \bb{F}) \cong \bb{F}$ denote the canonical generator.  The K\"unneth theorem for intersection homology (here we need field coefficients, see\cite[Thm.~4]{King}) provides an identification
$$IH^{\ol p}_{2d}(X \times \PP^t, \bb{F}) \cong \bigoplus_{j=0}^{d} IH_{2(d-j)}(X, \bb{F}) \cdot [\PP^j].$$
This isomorphism is functorial with respect to $X \times \PP^{t-1} \to X \times \PP^t$ and hence yields an identification $IH^{\ol p}_{2d}(X \times \PP^t, \bb{F}) / IH^{\ol p}_{2d}(X \times \PP^{t-1}, \bb{F}) \cong IH^{\ol p}_{2(d-t)}(X, \bb{F})$.  Altogether we obtained a map $H^{2t,t,\ol p}(X) \to IH^{\ol p}_{2(d-t)}(X, \bb{F})$ as desired.
\end{proof}

\begin{prop}
\label{prop:new compat}
Via the cycle class map described in Lemma $\ref{lem:cohom cycle class}$, the pairing in Proposition in $\ref{cupproduct}$ is compatible with the pairing in intersection homology.  In other words, the following diagram is commutative:

$$\xym{ H^{2s,s,\ol p}(X) \otimes H^{2t,t,\ol q}(X)  \ar[r]^-{\cup} \ar[d]_{c \otimes c} &   H^{2(s+t), s+t,\ol p + \ol q}(X) \ar[d]^c \\
IH^{\ol p}_{2(d-s)}(X, \bb{F}) \otimes IH^{\ol q}_{2(d-t)}(X, \bb{F}) \ar[r]  & IH^{\ol p + \ol q}_{2(d -s -t)}(X, \bb{F}) } $$

\end{prop}

\begin{proof}
In the smooth locus of $X$, the join maps to the cup product of cohomology classes \cite[Prop.~6.3]{FLcocycle}.
Pairs of generalized cocycles intersect properly in each stratum, and the intersection homology pairing between chains intersecting properly in each stratum is determined by the cup product of the corresponding cohomology classes in the smooth locus \cite[2.1]{GM1}.   Therefore it suffices to show the identification $IH^{\ol p}_{2d}(X \times \PP^t, \bb{F}) / IH^{\ol p}_{2d}(X \times \PP^{t-1}, \bb{F}) \cong IH^{\ol p}_{2(d-t)}(X, \bb{F})$ is compatible with products.  But this identification may be described as the pull-back (\cite[5.4.3]{GM2}) $p_1^* : IH^{\ol p}_{2(d-t)}(X, \bb{F}) \to IH^{\ol p}_{2d}(X \times \PP^t, \bb{F})$ followed by the canonical projection onto $IH^{\ol p}_{2d}(X \times \PP^t, \bb{F}) / IH^{\ol p}_{2d}(X \times \PP^{t-1}, \bb{F})$, and both of these maps are compatible with intersection pairings. \end{proof}

\bibliography{intlawhom}{}
\bibliographystyle{plain}
\end{document}